\documentclass[12pt,reqno]{amsart}
\usepackage{latexsym}
\usepackage{amssymb}
\usepackage{amscd}

\usepackage{comment}

\numberwithin{equation}{section}

\newtheorem{theorem}{Theorem}
\newtheorem{proposition}[theorem]{Proposition}
\newtheorem{lemma}[theorem]{Lemma}
\newtheorem{corollary}[theorem]{Corollary}

\theoremstyle{remark}

\newtheorem{remarknu}[theorem]{Remark}
\newtheorem*{note}{Note}

\def\al{\alpha}
\def\be{\beta}
\def\de{\delta}
\def\ga{\gamma}
\def\ep{\varepsilon}

\def\Om{\Omega}

\def\Z{{\mathbb Z}}

\def\coef#1{\left\langle#1\right\rangle}

\def\Pol{\operatorname{Pol}}
\def\fl#1{\left\lfloor#1\right\rfloor}

\begin{document}
\title[Motzkin numbers modulo powers of $2$]%
{Motzkin numbers and related sequences modulo powers of $2$}
\author[C. Krattenthaler and 
T.\,W. M\"uller]{C. Krattenthaler$^{\dagger}$ and
T. W. M\"uller$^*$} 

\address{$^{\dagger*}$Fakult\"at f\"ur Mathematik, Universit\"at Wien,
Oskar-Morgenstern-Platz~1, A-1090 Vienna, Austria.
WWW: {\tt http://www.mat.univie.ac.at/\lower0.5ex\hbox{\~{}}kratt}.}

\address{$^*$School of Mathematical Sciences, Queen Mary
\& Westfield College, University of London,
Mile End Road, London E1 4NS, United Kingdom.%\newline
%WWW: \tt http://www.maths.qmw.ac.uk/\~{}twm/.
}

\thanks{$^\dagger$Research partially supported by the Austrian
Science Foundation FWF, grants Z130-N13 and S50-N15,
the latter in the framework of the Special Research Program
``Algorithmic and Enumerative Combinatorics"\newline\indent
$^*$Research supported by Lise Meitner Grant M1661-N25 of the Austrian
Science Foundation FWF}



\begin{abstract}
We show that the generating function
$\sum_{n\ge0}M_n\,z^n$ for Motzkin numbers $M_n$, when
coefficients are reduced modulo a given power of~$2$, can be
expressed as a polynomial in the basic series
$\sum _{e\ge0} ^{}  {z^{4^e}}/( {1-z^{2\cdot 4^e}})$ with
coefficients being Laurent polynomials in $z$ and $1-z$.
We use this result to determine $M_n$ modulo~$8$ in terms
of the binary digits of~$n$, thus improving, respectively
complementing earlier results by Eu, Liu and Yeh
[{\it Europ.\ J. Combin.} {\bf 29} (2008), 1449--1466]
and by Rowland and Yassawi 
[{\it J.~Th\'eorie Nombres Bordeaux} {\bf 27} (2015), 245--288]. 
Analogous results are also shown to hold for related
combinatorial sequences, namely for the Motzkin prefix numbers,
Riordan numbers, central trinomial coefficients, and for the sequence of
hex tree numbers. 
\end{abstract}
\maketitle

\section{Introduction}
The {\it Motzkin numbers} $M_n$ may be defined by 
(cf.\ \cite[Ex.~6.37]{StanBI})
\begin{equation} \label{eq:Motzkin} 
\sum_{n\ge0}M_n\,z^n=\frac {1-z-\sqrt{1-2z-3z^2}} {2z^2}.
\end{equation}
They have numerous combinatorial interpretations, see \cite[Ex.~6.38]{StanBI}.
The most basic one says that $M_n$ equals the number of
lattice paths from $(0,0)$ to $(n,0)$ consisting of steps taken from
the set
$\{(1,0),(1,1),(1,-1)\}$ never running below the $x$-axis.

Deutsch and Sagan \cite[Theorem~3.1]{DeSaAA} determined the parity of $M_n$.
Furthermore, they conjectured \cite[Conj.~5.5]{DeSaAA} (a conjecture
which is in part also due to Amdeberhan)
that $M_n$ is divisible by $4$ if, and only if,
$$
n=(4i + 1)4^{j+1} - 1\quad \text{or}\quad  n = (4i + 3)4^{j+1} - 2,
$$
for some non-negative integers $i$ and $j$.
This conjecture prompted Eu, Liu and Yeh \cite{EuLYAB} to embark
on a more systematic study of Motzkin numbers modulo~$4$
and~$8$. They approached this problem by carefully analysing
a binomial sum representation for $M_n$ modulo~$4$ and ~$8$.
They succeeded in establishing the above conjecture
of Amdeberhan, Deutsch and Sagan. Moreover, they were able to characterise
the even congruence classes of $M_n$ modulo~$8$; see 
\cite[Theorem~5.5]{EuLYAB}
and Corollary~\ref{thm:EuLYAB} at the end of the present paper.
One outcome of this characterisation is that $M_n$ is never divisible
by~$8$, thus proving another conjecture of Deutsch and Sagan 
\cite[Conj.~5.5]{DeSaAA}.

As part of a general theory of diagonals of rational functions,
Rowland and Yassawi \cite[Sec.~3.2]{RoYaAA} are able to approach the problem
of determination of Motzkin numbers modulo powers of~$2$ from a
completely different point of view. Their general theory (and
algorithms) produce automata which output the congruence class of
diagonal coefficients of a given rational function modulo a 
given power of a prime number.
Since the Motzkin numbers $M_n$
can be represented as the diagonal coefficients of a certain
rational function, this theory applies, and consequently,
for any prime power~$p^\alpha$, an automaton can be produced which outputs
the congruence class of $M_n$ modulo~$p^\alpha$.
Rowland and Yassawi present the automaton for the determination of
$M_n$ modulo~$8$ in \cite[Fig.~4]{RoYaAA}.

The main purpose of this paper is to present yet another approach
to the determination of Motzkin numbers $M_n$ modulo powers of~$2$,
namely by the use of generating functions. Our approach is in line
with our earlier studies \cite{KaKMAA,KrMuAE,KrMuAL}, in which we
developed a generating function method for determining the
mod-$p^\alpha$ behaviour of sequences whose generating functions
satisfy algebraic differential equations. Application of this
method to a concrete class of problems starts with the choice
of a (usually transcendental) basic series. 
The goal then is
to show that the generating function of the sequence that we are
interested in can be expressed as a polynomial in this basic series
when the terms of the sequence are reduced modulo~$p^\alpha$.
Once this is achieved, one tries to extract a characterisation of the
congruence behaviour of the sequence under consideration modulo~$p^\alpha$ 
from this polynomial representation of the generating function.
However, the details --- such as the choice of basic series, or the
extraction of coefficients from its powers --- have to be adjusted
for each individual class of applications.

Here, we choose the series $\Om(z)$ defined by
\begin{equation} \label{eq:Omdef} 
\Om(z)=\sum _{e\ge0} ^{}\sum _{f\ge0} ^{}
z^{4^e(2f+1)}
=\sum _{e\ge0} ^{} \frac {z^{4^e}} {1-z^{2\cdot 4^e}}
\end{equation}
as basic series.
In plain words, the series $\Om(z)$ is the sum of all monomials $z^m$,
where $m$ runs through all positive integers with an even number of
zeroes to the far right of their binary representation.

Our first main result, Theorem~\ref{thm:Motzkin}, states that the
generating function $\sum_{n\ge0}M_n\,z^n$ 
for the Motzkin numbers, when coefficients are reduced modulo a given
power of~$2$, can be expressed as a polynomial in $\Om(z^4)$ 
with coefficients that are Laurent polynomials in $z$ and $1-z$. 

Theorem~\ref{thm:Motzkin} may be 
implemented, so that these polynomial expressions can be found
automatically modulo any given power of $2$. Thus, our result in
principle opens up the possibility of deriving congruences modulo
arbitrary $2$-powers for the sequence of Motzkin numbers, provided one
is able to systematically extract coefficients from powers of the
series $\Om(z^4)$. However, for exponents greater than $3$, this task
turns out to be difficult and rather involved, though possible in
principle. For this reason, we proceed differently. We use our
algorithm to obtain a congruence modulo $16$ exhibiting the generating
function of the Motzkin numbers as a polynomial of degree $7$ in
$\Om(z^4)$, and then drastically simplify the resulting expression
modulo $2^3=8$ at the expense of  introducing a certain
``error-series'' $E(z)$. The advantage is that, in this step, the
degree of $\Om(z^4)$ drops from $7$ to $1$. Extracting 
coefficients of powers of~$z$ from the resulting expression, we are then able
to go beyond the result \cite[Theorem~5.5]{EuLYAB}
of Eu, Liu and Yeh by providing a formula for the congruence class
of $M_n$ also in the case that $M_n$ is odd; see
Theorem~\ref{thm:M8}. As demonstrated in Corollary~\ref{cor:Mn1},
this result may conveniently be used to precisely characterise those~$n$ 
for which $M_n$ lies in a specified congruence class modulo~$8$.

In a sense, our generating function method and the automaton method
of \cite{RoYaAA} are complementary to each other. It is through
our method that we were led to the explicit characterisation
of the congruence behaviour of Motzkin numbers modulo~8 in
Theorem~\ref{thm:M8} and Corollary~\ref{cor:Mn1}, and to
the explicit characterisations of the congruence behaviour 
modulo~8 of Motzkin prefix numbers, of Riordan numbers,
of hex tree numbers, and of central trinomial coefficients 
in Theorems~\ref{thm:MP8}--\ref{thm:T8}, together with proofs
in each case.
It seems hard to {\it extract\/} these results from the automata
that one obtains by the theory in \cite{RoYaAA}. On the other hand,
as pointed out by a referee, once these results are {\it known},
they can be {\it verified\/} using the theory of automata on the basis of the
automata that one finds by means of the automaton theory of \cite{RoYaAA},
since the functions $s_2(n)$, $s_2^2(n)$, $e_2(n)$, $n_4$, $n_{K+1}$
that appear in these results are $2$-regular sequences.

\medskip
Our paper is organised as follows.
In the next section, we derive some properties of our basic series
$\Om(z)$ defined in \eqref{eq:Omdef} that we shall need in the
sequel. Section~\ref{sec:method} outlines our generating function
approach based on the series $\Om(z)$. 
Subsequently, in Section~\ref{sec:Motzkin}, we apply this
method to the generating function $M(z)=\sum_{n\ge0}M_n\,z^n$
for Motzkin numbers. This section contains our first main result,
Theorem~\ref{thm:Motzkin}, stating that $M(z)$, when coefficients
are reduced modulo a given power of~$2$, can be expressed as a
polynomial in $\Om(z)$. In Section~\ref{sec:M8}, we specialise
this result to the modulus $2^4=16$, see \eqref{eq:M16}.
This expression is then reduced modulo~$8$ and further simplified,
at the cost of introducing an ``error series" $E(z)$, defined
in \eqref{eq:Edef}. The final expression for the mod-$8$
reduction of $M(z)$ is presented in \eqref{eq:MM}.
The remaining task is to extract the coefficient of $z^n$ from
the right-hand side of \eqref{eq:MM}. In order to accomplish this,
we provide several auxiliary lemmas in Section~\ref{sec:aux}.
In Section~\ref{sec:Mn8}, making use of these lemmas, we are then able
to provide explicit formulae 
for the congruence class of $M_n$ modulo~$8$ in terms of the
binary digits of~$n$. This leads to our second main result which
is presented in Theorem~\ref{thm:M8}. In Corollary~\ref{cor:Mn1},
we demonstrate by means of an example how to characterise those~$n$
for which $M_n$ lies in a given congruence class modulo~$8$.
Our final section then collects together results analogous to
Theorem~\ref{thm:M8} for sequences related to Motzkin numbers,
namely for the Motzkin prefix numbers, Riordan numbers, hex tree
numbers, and the central trinomial coefficients; see
Theorems~\ref{thm:MP8}--\ref{thm:T8}. These results generalise 
earlier ones of Deutsch and Sagan \cite{DeSaAA} from modulus~$2$ to
modulus~$8$. 

\begin{note}
This paper is accompanied by a {\sl Mathematica} file
and a {\sl Mathematica} notebook so that an interested reader is
able to redo (most of) the computations that are presented in
this article. File and notebook are available at the article's
website
{\tt http://www.mat.univie.ac.at/\lower0.5ex\hbox{\~{}}kratt/artikel/2motzkin.html}.
\end{note}

\section{The power series $\Om(z)$}
\label{sec:Om}

Here we consider the formal power series $\Om(z)$ defined in
\eqref{eq:Omdef}.
This series is the principal character in the method for determining
congruences of recursive sequences modulo $2$-powers
that we describe in Section~\ref{sec:method}. 
This series is transcendental over
$\mathbb Z[z]$.\footnote{The quickest argument uses coefficient
  asymptotics for algebraic functions. The coefficients of $\Om(z)$
do not approach any limit. Since these coefficients are $0$ or $1$,
there is no asymptotic formula for them.
On the other hand, by \cite[Theorem~VII.8]{FlSeAA}, the coefficients
of algebraic power series {\it do} have asymptotic formulae (whose form
can even be described very specifically). Consequently, the series
$\Om(z)$ cannot be algebraic.} However, if the
coefficients of $\Om(z)$ are considered modulo a $2$-power $2^\ga$,
then $\Om(z)$ obeys a polynomial relation with coefficients that are
Laurent polynomials in $z$ and $1-z$. This is shown in
Proposition~\ref{prop:minpol} below. In the second part of this
section, we show that the derivative of $\Om(z)$, when coefficients
are reduced modulo a given $2$-power, can be expressed as a 
polynomial in $z$ and $(1-z)^{-1}$.
This is one of the crucial facts which
make the method described in Section~\ref{sec:method} work.

\medskip
Here and in the sequel,
given integral power series (or Laurent series) $f(z)$ and $g(z)$,
we write 
$$f(z)=g(z)~\text {modulo}~2^\ga$$ 
to mean that the coefficients
of $z^i$ in $f(z)$ and $g(z)$ agree modulo~$2^\ga$ for all $i$.

%We say that a polynomial $A(z,t)$ in 
%$z$ and $t$ is {\it minimal for the modulus
%$2^\ga$}, if it is monic (as a polynomial in $t$), 
%has coefficients which are Laurent polynomials in $z$ and $1-z$, satisfies
%$A(z,\Om(z))=0$~modulo~$2^\ga$, and there is no monic polynomial $B(z,t)$
%whose coefficients are Laurent polynomials in $z$ and $1-z$, whose
%$t$-degree is less than that of $A(z,t)$, and which satisfies
%$B(z,\Om(z))=0$~modulo~$2^\ga$.
%Minimal polynomials are not unique; see Remark~\ref{rem:min} below.

%Our results concerning minimal polynomials are weaker than the
%corresponding ones in \cite[Sec.~2]{KaKMAA}.

\begin{proposition} \label{prop:minpol}
We have
\begin{equation} \label{eq:Om^2A}
\Om^2(z)+\Om(z)-\frac {z} {1-z}=0\quad 
\text {\em modulo }2.
\end{equation}
\end{proposition}

\begin{proof}
It is straightforward to see that
$$
\Om^2(z)=\sum _{e\ge0} ^{}\sum _{f\ge0} ^{}
z^{2\cdot 4^e(2f+1)}\quad 
\text{modulo }2.
$$
In plain words, this is the sum of all monomials $z^m$,
where $m$ runs through all positive integers with an odd number of
zeroes to the far right of their binary representation. Consequently,
we have 
$$
\Om^2(z)+\Om(z)=\sum_{m\ge1}z^m=\frac {z} {1-z}\quad 
\text{modulo }2.
$$
This is equivalent to the claim.
\end{proof}

%\begin{remarknu}\label{rem:min}
%Minimal polynomials are highly non-unique: for example, 
%in view of \eqref{eq:Om^2A}, both
%$$
%\left(t^2+t-\frac {z} {1-z}\right)^2
%$$
%and
%$$
%\left(t^2+t-\frac {z} {1-z}\right)^2+2\left(t^2+t-\frac {z} {1-z}\right)
%$$
%are obviously minimal polynomials for the modulus $4$.
%\end{remarknu}

\begin{remarknu}
Clearly, Proposition~\ref{prop:minpol} implies that $\Om(z)$
is algebraic modulo {\it any} power of~$2$. In particular, we have
\begin{equation} \label{eq:N} 
\left(\Om^2(z)+\Om(z)-\frac {z} {1-z}\right)^N=0\quad 
\text {modulo }2^N,
\end{equation}
a relation that we shall use later on.
\end{remarknu}

We now turn our attention to the derivative of the basic series
$\Om(z)$.
By straightforward differentiation, we obtain
\begin{equation} \label{eq:Om'} 
\Om'(z)
= \sum_{e\ge0}^{}4 ^e \left(\frac {
z^{4^e-1}} {1-z^{2\cdot 4^e}}
+
\frac {2
z^{3\cdot 4^e-1}} {(1-z^{2\cdot 4^e})^2}\right).
\end{equation}
We would like to prove that $\Om'(z)$, when considered modulo a given
power of~$2$, can be expressed as a polynomial in $z$ and $(1-z)^{-1}$
with integer coefficients. The above relation does not quite achieve this,
since the denominators on the
right-hand side are not powers of $1-z$. However, the following is
true.

\begin{lemma} \label{lem:1+z}
For all non-negative integers $j$ and positive integers $\al$ and
$\be,$ we have
$$
\frac {1} {(1+z^{2^j})^\al}=\Pol^+_{j,\al,\be}\left(z,(1-z)^{-1}\right)
\quad \text {\em modulo }2^\be,
$$
where $\Pol^+_{j,\al,\be}\left(z,(1-z)^{-1}\right)$ is a polynomial in $z$ and
$(1-z)^{-1}$ with integer coefficients.
\end{lemma}

\begin{proof} We perform an induction with respect to $j+\be$.
For $j=0$ and arbitrary $\beta$, we have
$$
\frac {1} {(1+z)^\al}=
\frac {1} {(1-z+2z)^\al}=
(1-z)^{-\al}\sum_{k\ge0}\binom {\al+k-1}k 2^k\frac {z^k} {(1-z)^k}.
$$
When this is considered modulo~$2^\be$ for some fixed~$\be$,
the sum on the right-hand side becomes finite and is indeed 
a polynomial in $z$ and $(1-z)^{-1}$.

For the induction step with $j\ge1$, we write
\begin{align*}
\frac {1} {(1+z^{2^j})^\al}&=
\frac {1} {(1+z^{2^{j-1}})^{2\al}}
+\frac {1} {(1+z^{2^j})^\al}
-\frac {1} {(1+z^{2^{j-1}})^{2\al}}\\
&=
\frac {1} {(1+z^{2^{j-1}})^{2\al}}
+\frac {1} {(1+z^{2^{j-1}})^{2\al}(1+z^{2^j})^\al}
\left((1+z^{2^{j-1}})^{2\al}-(1+z^{2^j})^\al\right)\\
&=
\frac {1} {(1+z^{2^{j-1}})^{2\al}}
+\frac {1} {(1+z^{2^{j-1}})^{2\al}(1+z^{2^j})^\al}\\
&\kern4cm
\times
\left((1+2z^{2^{j-1}}+z^{2^{j}})^{\al}
-(1+z^{2^j})^\al\right)\\
&=
\frac {1} {(1+z^{2^{j-1}})^{2\al}}
+\frac {1} {(1+z^{2^{j-1}})^{2\al}(1+z^{2^j})^\al}\\
&\kern4cm
\times
\sum _{\ell=1} ^{\al}\binom\al\ell 
2^\ell\left(z^{2^{j-1}}\right)^{\ell}
\left(1+z^{2^j}\right)^{\al-\ell}\\
&=
\frac {1} {(1+z^{2^{j-1}})^{2\al}}
+
\sum _{\ell=1} ^{\al}\binom\al\ell 
\frac {2^\ell z^{\ell \cdot 2^{j-1}}} 
{(1+z^{2^{j-1}})^{2\al}(1+z^{2^j})^\ell}.
\end{align*}
Now the induction hypothesis can be applied, and yields
\begin{multline*}
\frac {1} {(1+z^{2^j})^\al}=
\Pol^+_{j-1,2\al,\be}\left(z,(1-z)^{-1}\right)\\
+
\sum _{\ell=1} ^{\al}\binom\al\ell 
{2^\ell z^{\ell \cdot 2^{j-1}}} 
(\Pol^+)^{2\al}_{j-1,2\al,\be-\ell}\left(z,(1-z)^{-1}\right)
(\Pol^+)^{\ell}_{j,\ell,\be-\ell}\left(z,(1-z)^{-1}\right)\\
\text {modulo }2^\be.
\end{multline*}
This completes the induction.
\end{proof}

\begin{lemma} \label{lem:1-z}
For all non-negative integers $j$ and positive integers $\al$ and
$\be,$ we have
$$
\frac {1} {(1-z^{2^j})^\al}=\Pol_{j,\al,\be}\left(z,(1-z)^{-1}\right)
\quad \text {\em modulo }2^\be,
$$
where $\Pol_{j,\al,\be}\left(z,(1-z)^{-1}\right)$ is a polynomial in $z$ and
$(1-z)^{-1}$ with integer coefficients.
\end{lemma}

\begin{proof}
We perform an induction on~$j$. Clearly, there is nothing to prove
for $j=0$.

For the induction step, we write
$$
\frac {1} {(1-z^{2^j})^\alpha} 
=\frac {1} {(1+z^{2^{j-1}})^\alpha} \cdot \frac1 {(1-z^{2^{j-1}})^\alpha} .
$$
By Lemma~\ref{lem:1+z}, we know that the first factor on the right-hand
is a polynomial in $z$ and $(1-z)^{-1}$, while the induction
hypothesis guarantees that the second factor is such a polynomial as well.
\end{proof}

Identity~\eqref{eq:Om'} and Lemma~\ref{lem:1-z} together imply 
our claim concerning $\Om'(z)$ (it should be observed that, modulo a
given $2$-power $2^\be$, the sum on the right-hand side of \eqref{eq:Om'}
is finite).

\begin{corollary} \label{cor:Om'}
The series $\Om'(z),$ when coefficients are reduced modulo
$2^\be$ for some fixed~$\be,$ 
can be expressed as a polynomial in $z$ and $(1-z)^{-1}$
with integer coefficients. 
\end{corollary}

\section{The method}
\label{sec:method}

We consider a (formal) differential equation of the form
\begin{equation} \label{eq:diffeq}
\mathcal P(z;F(z),F'(z),F''(z),\dots,F^{(s)}(z))=0,
\end{equation}
where $\mathcal P$ is a polynomial with integer coefficients, which has a
unique power series solution $F(z)$ with integer coefficients when
the equation is considered modulo any fixed power of~$2$.
(In the literature, series obeying a polynomial relation of the
form \eqref{eq:diffeq} are known as {\it differentially algebraic}
series; see, for instance, \cite{BeReAA}.)
In this
situation, we propose the following algorithmic approach to
determining $F(z)$ modulo a $2$-power $2^{2^\al}$,
for some positive integer $\al$. 
Let us fix a positive integer $\ga$. 
We make the Ansatz
\begin{equation} \label{eq:Ansatz}
F(z)=\sum _{i=0} ^{2^{\al+1}-1}a_i(z)\Om^i( z^\ga)\quad 
\text {modulo }2^{2^\al},
\end{equation}
where $\Om(z)$ is given by \eqref{eq:Omdef},
and where the $a_i(z)$'s
are (at this point) undetermined elements of
$\Z[z,z^{-1},(1- z^\ga)^{-1}]$.
Now we substitute \eqref{eq:Ansatz} into \eqref{eq:diffeq}, and
we shall gradually determine approximations $a_{i,\be}(z)$ to $a_i(z)$ such that
\eqref{eq:diffeq} holds modulo~$2^\be$, for $\be=1,2,\dots,2^\al$. 
To initiate the procedure, we consider the differential equation
\eqref{eq:diffeq} modulo~$2$, with
\begin{equation} \label{eq:Ansatz1}
F(z)=\sum _{i=0} ^{2^{\al+1}-1}a_{i,1}(z)\Om^i( z^\ga)\quad \text {modulo
}2.
\end{equation}
By Corollary~\ref{cor:Om'}, we know that $\Om'(z^\ga)$, when
considered modulo~$2$, is in $\Z[z,(1-z^\ga)^{-1}]$. 
Consequently, we see that
the left-hand side of \eqref{eq:diffeq} is a polynomial in
$\Om( z^\ga)$ with coefficients in 
$\Z[z,z^{-1},(1- z^\ga)^{-1}]$. 
We reduce powers $\Om^k( z^\ga)$ with $k\ge2^{\al+1}$ using
the relation 
\begin{equation} \label{eq:OmRel}
\left(\Om^{2}( z^\ga)+\Om( z^\ga)
-\frac {z^\ga} {1- z^\ga}\right)^{2^\al}=0\quad \text {modulo
}2^{2^\al},
\end{equation}
which is the special case $N=2^\al$ of \eqref{eq:N}.
Since, at this point, we are only interested in finding a solution to
\eqref{eq:diffeq} modulo~$2$, Relation~\eqref{eq:OmRel} simplifies to
\begin{equation} \label{eq:OmRel2}
\Om^{2^{\al+1}}( z^\ga)+\Om^{2^{\al}}( z^\ga)
-\frac {z^{\ga\cdot 2^\al}} {(1- z^\ga)^{2^{\al}}}=0\quad \text {modulo
}2.
\end{equation}
Now we compare coefficients of powers $\Om^k( z^\ga)$,
$k=0,1,\dots,2^{\al+1}-1$. This yields a
system of $2^{\al+1}$ (differential) equations (modulo~$2$)
for unknown functions $a_{i,1}(z)$ in
$\Z[z,z^{-1},(1- z^\ga)^{-1}]$, for $i=0,1,\dots,
2^{\al+1}-1$,
which may or may not have a solution. 

Provided we have already found 
functions $a_{i,\be}(z)$ in
$\Z[z,z^{-1},(1- z^\ga)^{-1}]$, $i=0,1,\dots,\break
2^{\al+1}-1$, 
for some $\be$ with $1\le \be\le 2^{\al}-1$, such that
\begin{equation} \label{eq:Ansatz2}
F(z)=\sum _{i=0} ^{2^{\al+1}-1}a_{i,\be}(z)\Om^i( z^\ga)
\end{equation}
solves \eqref{eq:diffeq} modulo~$2^\be$, we put 
\begin{equation} \label{eq:Ansatz2a}
a_{i,\be+1}(z):=a_{i,\be}(z)+2^{\be}b_{i,\be+1}(z),\quad 
i=0,1,\dots,2^{\al+1}-1,
\end{equation}
where the $b_{i,\be+1}(z)$'s are (at this point) undetermined 
elements in $\Z[z,z^{-1},(1- z^\ga)^{-1}]$. Next we substitute
\begin{equation} \label{eq:Ansatz2b} 
F(z)=\sum _{i=0} ^{2^{\al+1}-1}a_{i,\be+1}(z)\Om^i( z^\ga)
\end{equation}
in \eqref{eq:diffeq} and consider the result modulo~$2^{\be+1}$. 
Using Corollary~\ref{cor:Om'} and Lemma~\ref{lem:1-z},
we see that derivatives of $\Om(z^\ga)$, when considered
modulo~$2^{\be+1}$, can be expressed as 
polynomials in $z$ and $(1-z^\ga)^{-1}$. Consequently, we may
expand the left-hand side of \eqref{eq:diffeq} 
as a polynomial in $\Om( z^\ga)$ with
coefficients in $\Z[z,z^{-1},(1- z^\ga)^{-1}]$. Subsequently,
we apply again the reduction
using Relation~\eqref{eq:OmRel}. By comparing coefficients of powers
$\Om^k( z^\ga)$, $k=0,1,\dots,2^{\al+1}-1$,
we obtain a
system of $2^{\al+1}$ (differential) equations (modulo~$2^{\be+1}$)
for the unknown functions $b_{i,\be+1}(z)$,
$i=0,1,\dots,2^{\al+1}-1$,
which may or may not have a solution. If we manage to push this 
procedure through until $\be=2^\al-1$, then,
setting $a_i(z)=a_{i,2^\al}(z)$, $i=0,1,\dots,2^{\al+1}-1$,
the series $F(z)$ as given in \eqref{eq:Ansatz} is a solution to
\eqref{eq:diffeq} modulo~$2^{2^\al}$, as required.

We should point out that, in order for the method to be meaningful,
it is essential to assume that the differential equation \eqref{eq:diffeq},
when considered modulo an arbitrary $2$-power $2^e$,
determines the solution $F(z)$ {\it uniquely} modulo~$2^e$ (which we do).
Otherwise, our method could produce several different ``solutions,"
and it might be difficult to decide which of them actually 
corresponds to the series $F(z)$ we are interested in.

\section{Motzkin numbers}
\label{sec:Motzkin}

Let $M_n$ be the $n$-th {\it Motzkin number}, that is, the number of
lattice paths from $(0,0)$ to $(n,0)$ consisting of steps taken from
the set
$\{(1,0),(1,1),(1,-1)\}$ never running below the $x$-axis.
It is well-known (cf.\ \cite[Ex.~6.37]{StanBI})
that the generating function $M(z)=\sum_{n\ge0}M_n\,z^n$ is given by
\eqref{eq:Motzkin}
and hence satisfies the functional equation
\begin{equation} \label{eq:MotzkinEF}
z^2M^2(z)+(z-1)M(z)+1=0.
\end{equation}

\begin{theorem} \label{thm:Motzkin}
Let $\Om(z)$ be given by \eqref{eq:Omdef}{\em ,} 
and let $\al$ be some positive integer.
Then the generating function $M(z)=\sum_{n\ge0}M_n\,z^n$ 
for the Motzkin numbers,
reduced modulo~$2^{2^\al},$ 
can be expressed as a polynomial in $\Om(z^4)$ of the form
$$
M(z)=\sum _{i=0} ^{2\cdot 2^{\al}-1}a_{i}(z)\Om^{i}(z^4)\quad 
\text {\em modulo }2^{2^\al},
$$ 
where the coefficients $a_{i}(z),$ $i=0,1,\dots,2^{\al+1}-1,$ 
are Laurent polynomials in $z$ and $1-z$.
\end{theorem}

\begin{proof}
We apply a slight variation of 
the method from Section~\ref{sec:method}. We start by
making the Ansatz (compare with \eqref{eq:Ansatz1})
\begin{equation*} %\label{eq:AAnsatz}
M(z)=\sum _{i=0} ^{2^{\al+1}-1}a_i(z)\Om^i(z^4)\quad 
\text {modulo }2^{2^\al},
\end{equation*}
where $\Om(z)$ is given by \eqref{eq:Omdef},
and where the $a_i(z)$'s
are (at this point) undetermined Laurent polynomials in $z$ and
$1-z$.

For the ``base step" (that is, for $\be=1$), we claim that
$$
M_1(z)=\sum_{k=1}^{\al+1}\frac {z^{2^k-2}} {(1-z)^{2^k-1}}+
\frac {1-z} {z^2}\Om^{2^\al}(z^4)
$$
solves \eqref{eq:MotzkinEF} modulo~2. Indeed, substitution of $M_1(z)$
in place of $M(z)$ on the left-hand side of \eqref{eq:MotzkinEF} yields
the expression
\begin{multline*}
{z^2}\left(
\sum_{k=1}^{\al+1}\frac {z^{2\cdot(2^k-2)}} {(1-z)^{2\cdot(2^k-1)}}+
\frac {(1-z)^2} {z^4}\Om^{2^{\al+1}}(z^4)
\right)\\
+(z-1)\left(\sum_{k=1}^{\al+1}\frac {z^{2^k-2}} {(1-z)^{2^k-1}}+
\frac {1-z} {z^2}\Om^{2^\al}(z^4)\right)
+1\quad \text{modulo }2.
\end{multline*}
Now one uses the relation \eqref{eq:OmRel2} 
to ``get rid of" $\Om^{2^{\al+1}}(z^4)$. This leads
to the expression
\begin{multline*}
{z^2}\left(
\sum_{k=1}^{\al+1}\frac {z^{2\cdot(2^k-2)}} {(1-z)^{2\cdot(2^k-1)}}+
\frac {(1-z)^2} {z^4}\Om^{2^{\al}}( z^4)
+\frac {(1-z)^2} {z^4}\frac {z^{4\cdot 2^\al}} {(1- z^4)^{2^{\al}}}
\right)\\
+(z-1)\left(\sum_{k=1}^{\al+1}\frac {z^{2^k-2}} {(1-z)^{2^k-1}}+
\frac {1-z} {z^2}\Om^{2^\al}(z^4)\right)
+1\\
=
\sum_{k=1}^{\al+1}\frac {z^{2^{k+1}-2}} {(1-z)^{2^{k+1}-2}}+
\frac {(1-z)^2} {z^2}\Om^{2^{\al}}( z^4)
+\frac {z^{2^{\al+2}-2}} {(1- z)^{2^{\al+2}-2}}
\\
+\sum_{k=1}^{\al+1}\frac {z^{2^k-2}} {(1-z)^{2^k-2}}+
\frac {(1-z)^2} {z^2}\Om^{2^\al}(z^4)
+1
\quad \text{modulo }2,
\end{multline*}
which is indeed $0$ modulo~2.

\medskip
After we have completed the ``base step," we now proceed with the
iterative steps described in Section~\ref{sec:method}. We consider
the Ansatz \eqref{eq:Ansatz2}--\eqref{eq:Ansatz2b}, where the
coefficients $a_{i,\be}(z)$ are supposed to provide a solution
$M_{\be}(z)=\sum _{i=0} ^{2^{\al+1}-1}a_{i,\be}(z)\Om^i(z^4)$ to
\eqref{eq:MotzkinEF} modulo~$2^\be$. This Ansatz, substituted in
\eqref{eq:MotzkinEF}, produces the congruence
\begin{equation*}
z^2M_{\be}^2(z)+(z-1)M_{\be}(z)
+2^\be(z-1)\sum _{i=0} ^{2^{\al+1}-1}b_{i,\be+1}(z)\Om^i(z^4)
+1=0
\quad 
\text {modulo }2^{\be+1}.
\end{equation*}
By our assumption on $M_{\be}(z)$, we may divide by $2^\be$.
Comparison of powers of $\Om(z^4)$ then yields a system of congruences
of the form
$$
(1-z)b_{i,\be+1}(z)+\text {Pol}_i(z)=0\quad 
\text {modulo }2,\quad \quad 
i=0,1,\dots,2^{\al+1}-1,
$$
where $\text {Pol}_i(z)$, $i=0,1,\dots,2^{\al+1}-1$, are in
$\Z[z,z^{-1},(1-z)^{-1}]$. This system being trivially
uniquely solvable, we have proved that, for an arbitrary positive
integer $\al$, our (variant of the) algorithm of
Section~\ref{sec:method} will produce a solution $M_{{2^{\al+1}}}(z)$ 
to\break \eqref{eq:MotzkinEF} modulo~$2^{2^\al}$ which is a
polynomial in $\Om(z^4)$ with coefficients that are
in\break $\Z[z,z^{-1},(1-z)^{-1}]$.
\end{proof}

\section{The generating function for Motzkin numbers modulo $8$}
\label{sec:M8}

We have implemented the algorithm described in the proof of
Theorem~\ref{thm:Motzkin}.\footnote{The {\sl Mathematica} input
files are freely available; see the Note at the end of the
Introduction.} 
If we apply it with $\al=2$, then we obtain
\begin{multline} \label{eq:M16}
\sum_{n\ge0}M_n\,z^n
=
\frac{8 z^{15}+8 z^{11}+8 z^9+8 z^8+8 z^5}
      {(1-z)^{19}}
+\frac{4 z^{10}+12 z^9+12 z^7}
      {(1-z)^{13}}\\
+\frac{2 z^{10}+4 z^8+12 z^7+2 z^6+12 z^5+2 z^4}
      {(1-z)^{11}}
+\frac{z^6}{(1-z)^7}
+\frac{z^2}{(1-z)^3}
+\frac{1}{1-z}\\
+ \frac{4 z^{10}
   }{(1-z)^{11}}\Om(z)
+\left(\frac{8
   z^{10}}{(1-z)^{13}}+\frac{4
   z^8}{(1-z)^9}+\frac{2
   z^6}{(1-z)^7}\right) \Om^2(z)\\
+\left(\frac{8
   z^8}{(1-z)^9}+\frac{8
   z^6}{(1-z)^9}+\frac{8
   z^4}{(1-z)^9}+\frac{4
   z^2}{(1-z)^7}\right) \Om^3(z)\\
+\left(
\frac{8 z^{10}}{(1-z)^{11}}
+\frac{8( z^9+ z^8+ z^6+ z^4+ z+1)}
      {(1-z)^{11}}
+\frac{4 z^4+12 z+12 z^{-1}}
      {(1-z)^5}\right.\\\left.
+\frac{2 z^2+12 z+6+12 z^{-1}+2 z^{-2}}
      {(1-z)^3}
+\frac{15}{z}+\frac{1}{z^2}
\right) \Om^4(z)\\
+\left(
\frac{8 z^2}{(1-z)^3}
+\frac{4 z^{-2}}{(1-z)^3}\right) \Om^5(z)
+\left(
\frac{8 (z^3+1+ z^{-1})}{(1-z)^5}
+\frac{4 z^{-2}}{(1-z)}
+\frac{2}{z^2}
+\frac{14}{z}\right)
   \Om^6(z)\\
+\left(\frac{4}{z^2}+\frac{12}{z}+\frac{8}{1-z}\right)
   \Om^7(z)
\quad \quad \text{modulo }16.
\end{multline}

We are aiming at congruences for the Motzkin numbers
modulo~$8$, thus we need to reduce the above expression
modulo~$8$. In order to do this efficiently, we introduce the
error series $E(z)$ defined by the relation
\begin{equation} \label{eq:OmE} 
\Om^2(z) + \Om(z) - \frac {z} {1 - z}-2E(z)=0.
\end{equation}
A straightforward computation shows that, explicitly, the series
$E(z)$ is given by
\begin{equation} \label{eq:Edef}
E(z)=\underset{(e_1,f_1)\prec(e_2,f_2)}{\sum_{e_1,f_1,e_2,f_2\ge0} }
z^{4^{e_1}(2f_1+1)+4^{e_2}(2f_2+1)}.
\end{equation}
Here, the symbol $\prec$ refers to the lexicographic order on pairs of
integers, i.e., $(e_1,f_1)\prec(e_2,f_2)$ if, and only if,
$e_1<e_2$, or $e_1=e_2$ and $f_1<f_2$.

We combine three reductions: (1) powers of $\Om(z^4)$ are reduced by
means of \eqref{eq:N} with $N=3$; (2) powers of $\Om(z^4)$ are further
reduced using Relation~\eqref{eq:OmE} (at the cost of introducing the
error series $E(z)$); (3) coefficients are reduced modulo~$8$.
The final result is
\begin{multline} \label{eq:MM}
M(z)=  \left(\frac{4 
   (z+1)}{z^2}E(z^4)+\frac{7}{z^2}+\frac{1}{z}\right.\\
\left.+\frac{2 z^7+2
   z^6+6 z^5+6 z^4+6 z^3+6 z^2+2
   z+2}{1-z^8}\right)\Om(z^4)\\
+\frac{4
   E(z^8) (z+1)}{z^2}-2 E(z^4)
   \left(\frac{7}{z^2}+\frac{5}{z}+\frac{2}
   {1-z}\right)\\
+\frac{\begin{matrix}
(z^{15}+5 z^{14}+7
   z^{13}+3 z^{12}+3 z^{11}+3 z^{10}+5
   z^9+5 z^8\kern2cm\\\kern3cm+z^7+5 z^6+3 z^5+7 z^4+3 z^3+3
   z^2+z+1)\end{matrix}}{1-z^{16}}\\
\quad \quad \text{modulo }8.
\end{multline}
Here, we also used that $E^2(z^4)=E(z^8)$~modulo~$2$. (This explains
the occurrence of $E(z^8)$ in the above expression.)

In order to proceed, we need to know how to extract coefficients
of~$z^n$ from the series
$$
E(z^4)\Om(z^4),\
\Om(z^4),\
\frac {1} {1-z^8}\Om(z^4),\
E(z^4),\
\text{and }\frac {1} {1-z}E(z^4).
$$
The next section is devoted to solving these problems.

\section{Coefficient extraction}
\label{sec:aux}

We begin with coefficient extraction from the error series $E(z)$.
In all what follows, coefficients $n_i$ will be assumed to be
either 0 or~1.

\begin{lemma} \label{lem:E}
For all non-negative integers $n,$ we have
\begin{equation} \label{eq:E} 
\coef{z^n}E(z)\equiv
\begin{cases} 
\chi(e>0)(1+2n_1)+n_2+n_3+3n_4+n_5+3n_6+n_7+\cdots \pmod4,\\
&\kern-8cm
\text{if\/
}n=4^{e}(1+n_12+n_22^2+\dots+n_k2^k),\\
2\chi(e>0)+n_1+2n_2\pmod4,\\
&\kern-8cm
\text{if\/
}n=2\cdot4^{e}(1+n_12+n_22^2+\dots+n_k2^k),
\end{cases}
\end{equation}
where $\chi(\mathcal A)=1$ if the assertion $\mathcal{A}$ holds true,
and $\chi(\mathcal A)=0$ otherwise.
\end{lemma}

\begin{proof}
We have to count the number of quadruples
$(e_1,f_1;e_2,f_2)$, where $e_1,e_2,f_1,f_2$ 
are non-negative integers with $e_1<e_2$, or $e_1=e_2$ and $f_1<f_2$,
such that 
\begin{equation*} %\label{eq:ef2} 
4^{e_1}(2f_1+1)+4^{e_2}(2f_2+1)=n.
\end{equation*}
We may write
\begin{equation} \label{eq:ef2} 
n=4^{e_1}(2f_1+1)+4^{e_2}(2f_2+1)=4^{e_1}\big(2f_1+1+4^{e_2-e_1}(2f_2+1)\big).
\end{equation}

Let first $n=1+n_12+n_22^2+\dots+n_k2^k$. Then, 
quadruples $(e_1,f_1;e_2,f_2)$ satisfying \eqref{eq:ef2} must satisfy
$e_1=0<e_2$. For given $e_2$, the possible numbers $f_2$ such that
\eqref{eq:ef2} is satisfied with a suitable $f_1$ are $0,1,\dots,F$,
where 
$$F=\fl{\frac {1} {2}\left(\frac {n-1} {4^{e_2}}-1\right)}.$$
Thus, the number of above quadruples equals
\begin{align}
\notag
\sum_{e_2>0}&\fl{\frac {1} {2}\left(\frac {n-1} {4^{e_2}}+1\right)}
=\sum_{e_2\ge1}
\fl{\frac {1} {2}\left(\frac {n_12+n_22^2+\dots+n_k2^k} {4^{e_2}}+1\right)}\\
\notag
&=\sum_{e_2\ge1}
\fl{\frac {n_1} {2^{2e_2}}
+\frac {n_2} {2^{2e_2-1}}+\dots
+\frac {n_{2e_2}+1} {2}+n_{2e_2+1}+2n_{2e_2+2}+\dots+2^{k-2e_2-1}n_k
}\\
\notag
&=\sum_{e_2\ge1}
(n_{2e_2}+n_{2e_2+1}+2n_{2e_2+2})\pmod 4\\
&=n_2+n_3+3n_4+n_5+3n_6+n_7+3n_8+\cdots\pmod4.
\label{eq:n2n3B}
\end{align}
This is in agreement with the
first case in \eqref{eq:E} with $e=0$.

Next let  $n=4^e(1+n_12+n_22^2+\dots+n_k2^k)$ with $e\ge1$. Then the above
arguments apply again (with $e_1=e<e_2$). Furthermore, there are now also
quadruples $(e_1,f_1;e_2,f_2)$ with $e_1=e_2<e$ satisfying
\eqref{eq:ef2}. Modulo~4, the only ones which are relevant in our count
are those with $e_1=e_2=e-1$ and 
$$
2f_1+2f_2+2=4+n_18+\dots+n_k2^{k+2}.
$$
The number of pairs $(f_1,f_2)$ of non-negative integers with
$f_1<f_2$ satisfying this equation equals
$$
1+n_12+\dots+n_k2^{k}\equiv 1+2n_1\pmod 4.
$$
If this is added to the right-hand side of \eqref{eq:n2n3B},
the resulting expression is in agreement with the first case of 
\eqref{eq:E} with $e>0$.

The last remaining case is $n=2\cdot
4^e(1+n_12+n_22^2+\dots+n_k2^k)$. Here, quadruples $(e_1,f_1;e_2,f_2)$
satisfying \eqref{eq:ef2} must necessarily satisfy $e_1=e_2\le e$. It is
then easy to see that, modulo~$4$, the only quadruples relevant for
our count are those with $e_1=e_2=e-1$ (which only exist if $e>0$), 
whose number is congruent to
$2$~ modulo~$4$, and those with
$e_1=e_2=e$, whose number equals
$$
n_1+n_22+\dots+n_k2^{k-1}\equiv n_1+2n_2\pmod 4.
$$
Both cases together lead to
the claimed result corresponding to the second case in \eqref{eq:E}. 
\end{proof}

Next, we turn to coefficient extraction from the series $E(z^4)/(1-z)$.
\begin{lemma} \label{lem:E1-z}
For all non-negative integers $n$ with binary representation
$n=n_0+n_12+n_22^2+\cdots,$ we have
\begin{multline} \label{eq:E1-z} 
\coef{z^n}E(z^4)\frac {1} {1-z}\equiv
\sum_{e_2>0} 
(n'_{2e_2+2}+n'_{2e_2+3})
(n_2+n_3+n_4+n_5+\cdots+n_{2e_2+1})\\
+\sum_{i\ge1}(n_{2i+1}+1)n_{2i+2}
\pmod2,
\end{multline}
where $n-4=n'_0+n'_12+n'_24+\cdots$.
\end{lemma}

\begin{proof}
We have to count the number of quintuples
$(e_1,f_1;e_2,f_2;g)$, where $e_1,e_2,f_1,f_2,g$ 
are non-negative integers with $e_1<e_2$, or $e_1=e_2$ and $f_1<f_2$,
such that 
\begin{equation*} %\label{eq:ef3} 
4^{e_1+1}(2f_1+1)+4^{e_2+1}(2f_2+1)+g=n.
\end{equation*}
Given $e_2$, the range of possible $f_2$'s is $0,1,\dots,F_2$,
where 
$$F_2=\fl{\frac {1} {2}\left(\frac {n-4} {4^{e_2+1}}-1\right)}.$$
Given $e_2$, $f_2$, and $e_1<e_2$, the range of possible $f_1$'s is 
$0,1,\dots,F_1$, where 
$$F_1=\fl{\frac {1} {2}\left(\frac {n-4^{e_2+1}(2f_2+1)} 
{4^{e_1+1}}-1\right)}.$$
Thus, the number of above quintuples with $e_1<e_2$ equals
\begin{align}
\notag
\sum_{e_2>0}&
\sum_{f_2=0} ^{F_2}
\sum_{e_1=0} ^{e_2-1}
\fl{\frac {1} {2}\left(\frac {n-4^{e_2+1}(2f_2+1)} {4^{e_1+1}}+1\right)}\\
&=
\notag
\sum_{e_2>0}
\sum_{f_2=0} ^{F_2}
\sum_{e_1=0} ^{e_2-1}
\left(
\fl{\frac {1} {2}\left(\frac {n}
  {4^{e_1+1}}+1\right)}
-2\cdot4^{e_2-e_1-1}(2f_2+1)
\right)
\\
&\equiv
\notag
\sum_{e_2>0}
\sum_{f_2=0} ^{F_2}
\sum_{e_1=0} ^{e_2-1}
(n_{2e_1+2}+n_{2e_1+3})
\pmod 2\\
&\equiv
\notag
\sum_{e_2>0}
\sum_{f_2=0} ^{F_2}
(n_2+n_3+n_4+n_5+\cdots+n_{2e_2+1})
\pmod 2\\
\notag
&\equiv
\sum_{e_2>0}
(F_2+1)\cdot
(n_2+n_3+n_4+n_5+\cdots+n_{2e_2+1})\pmod2\\
&\equiv
\sum_{e_2>0} 
(n'_{2e_2+2}+n'_{2e_2+3})
(n_2+n_3+n_4+n_5+\cdots+n_{2e_2+1})
\pmod 2.
\label{eq:n2n3C}
\end{align}
where $n-4=n'_0+n'_12+n'_24+\cdots$.

On the other hand, if $e_1=e_2$, then we want to count all triples
$(e_1,f_1,f_2)$ of non-negative integers with $f_1<f_2$ such that
$$
2\cdot 4^{e_1+1}(f_1+f_2+1)\le n.
$$
Since the number of pairs $(x,y)$ of non-negative integers with $x<y$
such that $x+y=k$ equals $\fl{(k+1)/2}$, the number of above triples
equals 
\begin{equation} \label{eq:e1F} 
\sum_{e_1\ge0}\sum_{k=0}^{F} \fl{\frac {k+1} {2}}
\end{equation}
where
$$
F=\fl{\frac {n} {2\cdot4^{e_1+1}}}-1.
$$
If $m=m_0+m_12+m_24+\cdots$, then it is not difficult to see that
$\sum_{k=0}^{m} \fl{(k+1)/2}$ is odd if and only if
$m\equiv1$~(mod~4). In symbols,
\begin{equation} \label{eq:sumk} 
\sum_{k=0}^{m} \fl{\frac {k+1}
  {2}}=m_0(m_1+1)\pmod2.
\end{equation}
Thus, modulo~2, the sum in \eqref{eq:e1F} becomes
\begin{equation}
\sum_{e_1\ge0}\sum_{k=0}^{F} \fl{\frac {k+1} {2}}
\equiv (n_3+1)n_4+(n_5+1)n_6+\cdots\pmod2.
\label{eq:n2n3D}
\end{equation}
The sum of \eqref{eq:n2n3C} and \eqref{eq:n2n3D} then yields
\eqref{eq:E1-z}.
\end{proof}

Our next series to be considered is $\Om(z)/(1-z^2)$.

\begin{lemma} \label{lem:Omz^2}
For all non-negative integers $n$ with binary representation
$n=n_0+n_12+n_22^2+\cdots,$ we have
\begin{equation} \label{eq:Omz^2} 
\coef{z^n} \frac {1} {1-z^2}\Om(z)\equiv
\begin{cases} 
1+n_1+2n_2\pmod4,&\text{if $n$ is odd,}\\
n_2+n_3+3n_4+n_5+3n_6+n_7+\cdots\pmod4,&\text{if $n$ is even.}
\end{cases}
\end{equation}
\end{lemma}

\begin{proof}
We have to count the number of triples
$(e_1,f_1;g)$, where $e_1,f_1,g$ 
are non-negative integers such that 
\begin{equation*} %\label{eq:ef3} 
4^{e_1}(2f_1+1)+2g=n.
\end{equation*}

Let first $n=1+n_12+n_22^2+\dots+n_k2^k$. 
Then necessarily $e_1=0$, so that we want to count all non-negative
integers $f_1$ with $f_1\le (n-1)/2$. Modulo~4, the number $(n-1)/2$ is
congruent to $n_1+2n_2$. The number of possible $f_1$'s is by~1 larger, and
this observation explains the first case in \eqref{eq:Omz^2}.

Now let $n=n_12+n_22^2+\dots+n_k2^k$. Given $e_1>0$, the integer
$f_1$ may range from $0$ to $F$, where
$$F=\fl{\frac {1} {2}\left(\frac {n} {4^{e_1}}-1\right)}.$$
Thus, the number of possible triples $(e_1,f_1;g)$ is
\begin{align*}
\sum_{e_2>0}&\fl{\frac {1} {2}\left(\frac {n} {4^{e_1}}+1\right)}\\
&=\sum_{e_2>0}
\fl{\frac {n_1} {2^{2e_1}}
+\frac {n_2} {2^{2e_1-1}}+\dots
+\frac {n_{2e_1}+1} {2}+n_{2e_1+1}+2n_{2e_1+2}+\dots+2^{k-2e_1-1}n_k
}\\
&\equiv\sum_{e_2>0}(n_{2e_1}+n_{2e_1+1}+2e_{2e_1+2})\pmod4\\
&\equiv
n_2+n_3+3n_4+n_5+3n_6+n_7+\cdots\pmod4,
\end{align*}
as desired.
\end{proof}

Finally, we address coefficient extraction from the product
$E(z)\Om(z)$.

\begin{lemma} \label{lem:OmE}
Let $n$ be a non-negative integer.

\smallskip
{\em(1)} If $n=4^e(1+n_12+n_22^2+\dots+n_k2^k),$ then
\begin{multline} \label{eq:OmE1} 
\coef{z^n} E(z)\,\Om(z)\equiv
(n_1+1)n_2+\sum_{i\ge1}n_{2i+1}(n_{2i+2}+1)
+\chi(e\ge1)\cdot(n_2+n_3+\cdots)\\
+\sum_{e_3>0} 
(n'_{2e_3}+n'_{2e_3+1})
(n_2+n_3+n_4+n_5+\cdots+n_{2e_3-1})
\pmod2,
\end{multline}
where $n4^{-e}-4=1+n'_12+n'_24+\cdots$.

\smallskip
{\em(2)} If $n=2\cdot 4^e(1+n_12+n_22^2+\dots+n_k2^k),$ then
\begin{equation} \label{eq:OmE2} 
\coef{z^n} E(z)\,\Om(z)\equiv
\chi(e\ge1)+(n_{1}+n_{2})
+(n_{1}+n_{2}+n_3+n_4+\cdots)(1+n_1)
\pmod2.
\end{equation}
\end{lemma}

\begin{proof}
We have to count the number of sextuples
$(e_1,f_1;e_2,f_2;e_3,f_3)$, where $e_1,e_2,f_1,f_2,\break e_3,f_3$ 
are non-negative integers with $(e_1,f_1)\prec (e_2,f_2)$,
such that 
\begin{equation} \label{eq:ef3} 
4^{e_1}(2f_1+1)+4^{e_2}(2f_2+1)+4^{e_3}(2f_3+1)=n.
\end{equation}
We need this number only modulo~2.
By pairing up sextuples $(e_1,f_1;e_2,f_2;e_3,f_3)$ with
$(e_1,f_1)\prec (e_3,f_3)\prec (e_2,f_2) $ with those
sextuples with $(e_3,f_3)\prec (e_1,f_1)\prec (e_2,f_2) $,
one sees that we may equivalently count all sextuples
$(e_1,f_1;e_2,f_2;e_3,f_3)$ with
$(e_1,f_1)\preceq (e_2,f_2)\preceq (e_3,f_3) $, but not all three
equal, satisfying \eqref{eq:ef3}. In the remainder of this proof,
we shall always assume these conditions.

\medskip
{\sc Case 1:}
 $n=4^e(1+n_12+n_22^2+\dots+n_k2^k)$. 
Here, we have to consider all possible cases for solutions to \eqref{eq:ef3}
to exist: (a) $e_1=e_2=e_3=e$; (b) $0\le e_1=e<e_2=e_3$; (c)
$0\le e_1=e<e_2<e_3$; (d) $0\le e_1=e_2<e_3$. 

\smallskip
(a) Writing $n=4^eu$, we see that \eqref{eq:ef3} reduces to
\begin{equation} \label{eq:f1f2f3} 
f_1+f_2+f_3=\frac {u-3} {2},
\end{equation}
with $f_1\le f_2\le f_3$ but not all three equal. Let us denote the
number of triples $(f_1,f_2,f_3)$ with $f_1+f_2+f_3=N$, $f_1\le f_2\le
f_3$, not all three equal, by $f(N)$. In order to compute $f(N)$,
we face a classical partition problem, namely the problem of counting
all integer partitions of $N$ with at most three parts, not all of
them equal. It is then folklore (cf.\ \cite[Sec.~3.2]{AndrAF}) that,
for the generating function we have
$$
\sum_{N=0}^\infty f(N)\,z^N=\frac {1} {(1-z)(1-z^2)(1-z^3)}-\frac {1} {1-z^3}
=\frac {z+z^2-z^3} {(1-z)(1-z^2)(1-z^3)}.
$$
Modulo~2, this reduces to
$$
\sum_{N=0}^\infty f(N)\,z^N=\frac {z} {1-z^4}
\quad \quad \text{modulo }2.
$$
In other words, $f(N)\equiv1$~(mod~2) if, and only if,
$N\equiv1$~(mod~4). Returning to our problem of counting triples
satisfying \eqref{eq:f1f2f3} modulo~2, this means that the number of the above
triples is odd if, and only if, $u\equiv5$~(mod~8). Hence, a slick way to
express this number of triples modulo~2 in terms of the binary digits of $u$ is
\begin{equation} \label{eq:(a)} 
n_0 (n_1+1)n_2=(n_1+1)n_2\pmod2.
\end{equation}

\smallskip
(b) With the convention $n=4^eu$ from above, we see that we have to
count all triples $(e_2,f_2,f_3)$ with $e_2>e$ and $f_2\le f_3$ satisfying
$$
2\cdot4^{e_2-e}(f_2+f_3+1)\le u-1.
$$
Since the number of pairs $(x,y)$ of non-negative integers with $x\le y$
such that $x+y=k$ equals $\fl{(k+2)/2}$, the number of above triples
equals 
\begin{equation} \label{eq:e1FA} 
\sum_{e_2>e}\sum_{k=0}^{F} \fl{\frac {k+2} {2}}
\end{equation}
where
$$
F=\fl{\frac {u-1} {2\cdot4^{e_2-e}}}-1.
$$
Using \eqref{eq:sumk} again, 
the sum in \eqref{eq:e1FA}, when reduced modulo~2, becomes
\begin{equation}
\sum_{e_2>e}\sum_{k=0}^{F} \fl{\frac {k+2} {2}}
\equiv n_3(n_4+1)+n_5(n_6+1)+\cdots\pmod2.
\label{eq:n2n3F}
\end{equation}

\smallskip
(c) With the same notation as above,
here we have to count all quadruples\break $(e_2,f_2;e_3,f_3)$ with
$e<e_2<e_3$ satisfying
$$
4^{e_2-e}(2f_2+1)+4^{e_3-e}(2f_3+1)\le u-1.
$$
Given $e_3$ with $e_3>e$, the range of possible $f_3$'s is $0,1,\dots,F_3$,
where 
$$F_3=\fl{\frac {1} {2}\left(\frac {u-5} {4^{e_3-e}}-1\right)}.$$
Given $e_3$, $f_3$, and $e_2$ with $e<e_2<e_3$, 
the range of possible $f_2$'s is $0,1,\dots,F_2$, where 
$$F_2=\fl{\frac {1} {2}\left(\frac {u-4^{e_3-e}(2f_3+1)-1} 
{4^{e_2-e}}-1\right)}.$$
Thus, the number of above quadruples with $e_2<e_3$ equals
\begin{align}
\notag
\sum_{e_3>e}&
\sum_{f_3=0} ^{F_3}
\sum_{e_2=e+1} ^{e_3-1}
\fl{\frac {1} {2}\left(\frac {u-4^{e_3-e}(2f_3+1)-1} {4^{e_2-e}}+1\right)}\\
&=
\notag
\sum_{e_3>e}
\sum_{f_3=0} ^{F_3}
\sum_{e_2=e+1} ^{e_3-1}
\left(
\fl{\frac {1} {2}\left(\frac {u-1}
  {4^{e_2-e}}+1\right)}
-2\cdot4^{e_3-e_2-1}(2f_3+1)
\right)
\\
&\equiv
\notag
\sum_{e_3>e}
\sum_{f_3=0} ^{F_3}
\sum_{e_2=e+1} ^{e_3-1}
(n_{2e_2-2e}+n_{2e_2-2e+1})
\pmod 2\\
&\equiv
\notag
\sum_{e_3>e}
\sum_{f_3=0} ^{F_3}
(n_2+n_3+n_4+n_5+\cdots+n_{2e_3-2e-1})
\pmod 2\\
\notag
&\equiv
\sum_{e_3>e}
(F_3+1)\cdot
(n_2+n_3+n_4+n_5+\cdots+n_{2e_3-2e-1})\pmod2\\
&\equiv
\sum_{e_3>e} 
(n'_{2e_3-2e}+n'_{2e_3-2e+1})
(n_2+n_3+n_4+n_5+\cdots+n_{2e_3-2e-1})
\pmod 2,
\label{eq:(c)}
\end{align}
where $u-4=1+n'_12+n'_24+\cdots$.

\smallskip
(d) Here, Equation \eqref{eq:ef3} reduces to
\begin{equation} \label{eq:2u} 
(f_1+f_2+1)+2\cdot 4^{e_3-e_1-1}(2f_3+1)=2\cdot4^{e-e_1-1} u.
\end{equation}
From this equation it is obvious that $e_1\le e-1$.
Since the second term on the left-hand side is divisible by~$2$
because of $e_3>e_1$, the sum $f_1+f_2+1$ must also be divisible
by~$2$. We need not consider the case where $f_1+f_2+1\equiv0$~(mod~4) 
since then the number of possible pairs $(f_1,f_2)$ with $f_1\le f_2$
is even.
Thus, we may assume that $f_1+f_2+1\equiv2$~(mod~4), in which case
the number of possible pairs $(f_1,f_2)$ with $f_1\le f_2$ is odd.
Inspecting
\eqref{eq:2u} under this assumption again, we conclude that $e_3\ge e+1$.

The counting problem that remains to be solved (modulo~2) hence is
to count all pairs $(e_3,f_3)$ with $e_3\ge e+1$ satisfying
\begin{equation} \label{eq:eu2} 
4^{e_3}(2f_3+1)\le 4^{e} \cdot u.
\end{equation}
Given $e_3$, the integer $f_3$ ranges between $0$ and $F_3$, where 
$$
F_3=\fl{\frac {1} {2}\left(\frac {u} {4^{e_3-e}}-1\right)}.
$$
The number of pairs $(e_3,f_3)$ satisfying
\eqref{eq:eu2} then is
\begin{align} 
\notag
\sum_{e_3>e}(F_3+1)
&=
\sum_{e_3>e}\fl{\frac {1} {2}\left(\frac {u} {4^{e_3-e}}+1\right)}\\
\notag
&\equiv
\sum_{e_3>e}(n_{2e_3-2e}+n_{2e_3-2e+1})
\quad \quad \text{(mod 2)}\\
&\equiv
n_2+n_3+\cdots
\quad \quad \text{(mod 2)}.
\label{eq:(d)}
\end{align}
It should be noted that, since in the current case $0\le e_1\le e-1$,
we must have $e\ge1$ in order that \eqref{eq:(d)} can actually occur.
The contribution of Subcase~(d) to the final result therefore is
\begin{equation} \label{eq:(d2)} 
\chi(e\ge1)\cdot(n_2+n_3+\cdots)
\quad \quad \text{(mod 2)}.
\end{equation}

\smallskip
Adding up the individual contributions \eqref{eq:(a)},
\eqref{eq:n2n3F}, \eqref{eq:(c)}, and \eqref{eq:(d2)}, we arrive
at the right-hand side of \eqref{eq:OmE1}.

\medskip
{\sc Case 2:}
 $n=2\cdot 4^e(1+n_12+n_22^2+\dots+n_k2^k)$. 
In order to have solutions to \eqref{eq:ef3} (with 
$(e_1,f_1)\preceq (e_2,f_2)\preceq (e_3,f_3) $, but not all three
equal, as we assume throughout), we must have
$0\le e_1=e_2<e_3$. In that case, Equation~\eqref{eq:ef3} reduces to
\begin{equation} \label{eq:eu1} 
2\cdot 4^{e_1}(f_1+f_2+1)+4^{e_3}(2f_3+1)=n=2\cdot 4^e\cdot u,
\end{equation}
where $u$ is some odd number. 
%It is obvious from \eqref{eq:eu1}
%and $e_1<e_3$ that we must have $e_1\le e$. 

Let us for the moment fix the sum of $f_1+f_2+1$ to equal~$k$, say.
If this number is divisible by~$4$, say $k=4k'$, then
the number of possible pairs $f_1,f_2$ with $f_1\le f_2$ and
$f_1+f_2+1=4k'$ is
$\fl{(4k'+1)/2}=2k'$, which is even. Since we are only interested in
numbers of solutions modulo~2, we may disregard these cases.

There are two possibilities which remain to be considered: either
\begin{equation} \label{eq:f1f2A} 
f_1+f_2+1=2u_0+1
\end{equation}
for some $u_0\ge0$, or
\begin{equation} \label{eq:f1f2B} 
f_1+f_2+1=2(2u_0+1)
\end{equation}
for some $u_0\ge0$.

\smallskip
In the first case, that is, if \eqref{eq:f1f2A} holds, then,
from \eqref{eq:ef3}, we obtain
$$
2\cdot4^{e_1}(2u_0+1)+4^{e_3}(2f_3+1)=n=2\cdot 4^e\cdot u,
$$
for some odd positive integer~$u$. Because of $e_1<e_3$, this implies
$e_1=e$. 

Given $e_3$ with $e_3>e$, the range of possible $f_3$'s is $0,1,\dots,F_3$,
where 
$$F_3=\fl{\frac {1} {2}\left(\frac {n} {4^{e_3}}-1\right)}.$$
Given $e_3>e$ and $f_3$, the numbers $e_1$ and $u_0$ are uniquely
determined (recall that $e_1=e$). Since 
the number of pairs $(f_1,f_2)$ with $f_1\le f_2$ 
satisfying \eqref{eq:f1f2A} equals 
$\fl{\frac {1} {2}\big((2u_0+1)+1\big)}=u_0+1$,
the number of solutions to \eqref{eq:ef3} which we have to consider in the
current case is
\begin{align}
\notag
\sum_{e_3>e}&
\sum_{f_3=0} ^{F_3}(u_0+1)
=
\sum_{e_3>e}
\sum_{f_3=0} ^{F_3}
\fl{\frac {1} {2}\left(\frac {n-4^{e_3}(2f_3+1)} {2\cdot
    4^{e}}+1\right)}
\\
&=
\notag
\sum_{e_3>e}
\sum_{f_3=0} ^{F_3}
\left(
\fl{\frac {1} {2}\left(\frac {n}
  {2\cdot4^{e}}+1\right)}
-4^{e_3-e-1}(2f_3+1)
\right)
\\
&\equiv
\notag
\sum_{e_3>e}
(F_3+1)
\left(
\fl{\frac {1} {2}\left(\frac {n}
  {2\cdot4^{e}}+1\right)}
-\chi(e_3=e+1)
\right)
\pmod 2\\
&\equiv
\notag
\sum_{e_3>e}
(n_{2e_3-2e-1}+n_{2e_3-2e})
\left((1+n_1)
-\chi(e_3=e+1)
\right)
\pmod 2\\
&\equiv
(n_1+n_2+n_3+\cdots)(1+n_1)-(n_1+n_2)
\pmod 2.
\label{eq:n2n3E}
\end{align}

\smallskip
In the second case, that is, if \eqref{eq:f1f2B} holds, then the number
of possible pairs $(f_1,f_2)$ with $f_1\le f_2$ satisfying
\eqref{eq:f1f2B} equals $\fl{\frac {1} {2}\big(2(2u_0+1)+1\big)}=2u_0+1$,
which is odd. Since we are only interested in the number of solutions
to \eqref{eq:ef3} modulo~2, we may therefore simply continue with
$u_0$, and determine the number of quadruples $(e_1,u_0,e_3,f_3)$
for which
$$
4^{e_1+1}(2u_0+1)+4^{e_3}(2f_3+1)=2\cdot 4^e\cdot u,
$$
for some odd positive integer~$u$. Because of $e_1<e_3$, it follows
that $e_1+1=e_3$. Thus, the above relation becomes
$$
u_0+f_3+1=4^{e_3-e}\cdot u.
$$
There is no relation between $u_0$ and $f_3$, hence the number of
pairs $(u_0,f_3)$ satisfying the above relation equals the right-hand
side, which is $4^{e_3-e}u$. This quantity is odd if, and only if,
we have $e_3=e$. Since $e_3=e_1+1\ge1$, 
the contribution of the current case is $\chi(e\ge1)$.
If we add this to the earlier contribution \eqref{eq:n2n3E}, then
we obtain the right-hand side of \eqref{eq:OmE2}.

\medskip
This completes the proof of the lemma.
\end{proof}

\section{Motzkin numbers modulo $8$}
\label{sec:Mn8}

Using the auxiliary results from Section~\ref{sec:aux} in
\eqref{eq:MM}, we are now able to provide explicit formulae for
the congruence class of~$M_n$ modulo~8 in terms of the binary digits
of~$n$. 

\begin{theorem} \label{thm:M8}
Let $n$ be a positive integer with binary expansion
$$
n=n_0+n_1\cdot2+n_2\cdot2^2+\cdots.
$$
Furthermore, let $K$ denote the least integer $\ge4$ such that
$n_K=0$. 
The Motzkin numbers $M_n$ satisfy the following congruences
modulo~$8${\em:}
$$M_n\equiv_8 \begin{cases} 
2 s_2^2(n) + 4 e_2(n) + 1,&\text{if\/ }n\equiv0~\text{\em(mod $16$),}\\
6 s_2^2(n) + 4 e_2(n) + 3,
&\text{if\/ }n\equiv1~\text{\em(mod $16$),}\\
4 s_2(n) + 6,&\text{if\/ }n\equiv2~\text{\em(mod $16$),}\\
4,&\text{if\/ }n\equiv3~\text{\em(mod $16$),}\\
6 s_2^2(n) + 4 e_2(n) + 3,&\text{if\/ }n\equiv4~\text{\em(mod $16$),}\\
2 s_2^2(n) + 4 e_2(n) + 5,&\text{if\/ }n\equiv5~\text{\em(mod $16$),}\\
6 s_2^2(n) + 4 n_4\, s_2(n) + 4 e_2(n) + 2 n_4 + 7,
&\text{if\/ }n\equiv6~\text{\em(mod $16$),}\\
2 s_2^2(n) + 4 n_4\, s_2(n) + 4 e_2(n) + 2 n_4 + 5,
&\text{if\/ }n\equiv7~\text{\em(mod $16$),}\\
2 s_2^2(n) + 4 e_2(n) + 1,&\text{if\/ }n\equiv8~\text{\em(mod $16$),}\\
6 s_2^2(n) + 4 e_2(n) + 3,
&\text{if\/ }n\equiv9~\text{\em(mod $16$),}\\
4,&\text{if\/ }n\equiv10~\text{\em(mod $16$),}\\
4 s_2(n) + 2,&\text{if\/ }n\equiv11~\text{\em(mod $16$),}\\
6 s_2^2(n) + 4 e_2(n) + 3,
&\text{if\/ }n\equiv12~\text{\em(mod $16$),}\\
2 s_2^2(n) + 4 e_2(n) + 5,&\text{if\/ }n\equiv13~\text{\em(mod $16$),}\\
(n_{K+1} + 1) (4 s_2(n) + 6), 
&\text{if\/ }n\equiv14~\text{\em(mod $16$) and $K$ is even,}\\
 2 n_{K+1} + 4 n_{K+1} s_2(n) + 2 s_2^2(n) + 4 s_2(n) + 4 e_2(n) + 
  7,\hskip-3cm\\
&\text{if\/ }n\equiv14~\text{\em(mod $16$) and $K$ is odd,}\\
6 n_{K+1} + 4 n_{K+1} s_2(n) + 4,
&\text{if\/ }n\equiv15~\text{\em(mod $16$) and $K$ is even,}\\ 
 2 n_{K+1} + 4 n_{K+1} s_2(n) + 2 s_2^2(n) + 4
 e_2(n) + 5,\hskip-1cm\\
&\text{if\/ }n\equiv15~\text{\em(mod $16$) and $K$ is odd}.
\end{cases}
$$
Here, $s_2(n)$ denotes the sum of the binary digits of~$n,$
while $e_2(n)$ denotes the number of pairs of successive digits
in the binary expansion of~$n,$ with both digits equalling~$1$.
The symbol $x\equiv_8y$ is short for $x\equiv y$~{\em(mod $8$)}.
\end{theorem}

\begin{proof}[Sketch of proof]
We have to discuss 16 different cases depending on the congruence
class of $n$ modulo~$16$. Since the arguments in all these cases
are very similar, we content ourselves with presenting the most
complex case in detail, leaving the remaining cases to the reader.

Let $n\equiv15$~(mod~16), and
let $L$ be the least integer $\ge K+2$ with $n_L\ne0$.
In other words, the binary representation of the number $n$ is of the form
\begin{equation} \label{eq:n2} 
(n)_2=
\dots\overset{L}{\overset{\downarrow}1}0\dots0n_{K+1}
\overset{K}{\overset{\downarrow}0}1\dots1
\raise-1pt\hbox to0pt{.\hss}
\raise2pt\hbox to0pt{.\hss}
\raise5pt\hbox to0pt{.\hss}
\raise8pt\hbox to0pt{.\hss}\hphantom{.}
1111.
\end{equation}
(The vertical dots separate significant digits from the four 1's on
the right making $n$ a number congruent to $15$
modulo~$16$. Otherwise, these dots may be safely ignored.)

Let first $K$ be even, say $K=2k$. For our specific~$n$, we extract the
coefficient of $z^n$ in \eqref{eq:MM}. The relevant terms are
\begin{multline} \label{eq:terms}
4z^{-1}E(z^4)\Om(z^4),\
z^{-1}\Om(z^4),\
\frac {6z^3} {1-z^8}\Om(z^4),\
\frac {2z^7} {1-z^8}\Om(z^4),\\
4z^{-1}E(z^8),\
-10z^{-1}E(z^4),\
-\frac {4} {1-z}E(z^4),\
\frac {z^{15}} {1-z^{16}}.
\end{multline}
Coefficient extraction for these terms is discussed in
Lemmas~\ref{lem:OmE}, 
\ref{lem:Omz^2},
\ref{lem:E},
and \ref{lem:E1-z},
respectively (ordered according to the order of the terms in \eqref{eq:terms}).
Adding all contributions, we obtain
\begin{multline} \label{eq:M15}
4(n_{2k+1}+1)n_{2k+2}+
4\sum_{i\ge k+1}n_{2i+1}(n_{2i+2}+1)+
4\sum_{i\ge 2k+2}n_i\\
+
4\sum_{e>k}(n'_{2e}+n'_{2e+1})(n_{2k+2}+n_{2k+3}+\dots+n_{2e-1})\\
+
1+6(2+2n_4)+2(n_4+n_5+3n_6+n_7+3n_8+n_9+\cdots)+
4(2+n_{2k+1}+2n_{2k+2})\\
+
6(1+2n_{2k+1}+n_{2k+2}+n_{2k+3}+3n_{2k+4}+n_{2k+5}+3n_{2k+6}+n_{2k+7}+\cdots)\\
+
4\sum_{e>k}(n_{2e+2}+n_{2e+3})(n_{2}+n_{3}+\dots+n_{2e+1})
+
4\sum_{i\ge1}(n_{2i+1}+1)n_{2i+2}+1,
\end{multline}
where the digits $n'_j$ are defined by
$$
(n+1)2^{-2k}-4=1+n'_{2k+1}2+n'_{2k+2}4+\cdots.
$$

In order to control the auxiliary digits $n'_j$,
we have to distinguish between $L$ (recall \eqref{eq:n2}) being
even or odd. If this is done, one sees that one may
simply drop the primes without changing the result.

Next, we observe that
$$
n_2+n_3+\dots+n_{2k+1}\equiv1+n_{2k+1}\pmod2.
$$
Using this, we may split the first sum in the last line of
\eqref{eq:M15} in the following way:
\begin{align} \notag
4\sum_{e>k}&(n_{2e+2}+n_{2e+3})(n_{2}+n_{3}+\dots+n_{2e+1})\\
\notag
&\equiv
4\sum_{e>k}(n_{2e+2}+n_{2e+3})(1+n_{2k+1})\\
\notag
&\kern2cm
+4\sum_{e>k}(n_{2e+2}+n_{2e+3})(n_{2k+2}+n_{2k+3}+\dots+n_{2e+1})
\pmod8\\
\notag
&\equiv
4(1+n_{2k+1})\sum_{i\ge 2k+4}n_i
+4\sum_{e>k}(n_{2e+2}+n_{2e+3})(n_{2k+2}+n_{2k+3}+\dots+n_{2e+1})\\
&\kern12cm
\pmod8.
\label{eq:n'n}
\end{align}
The last sum in this expression cancels with the sum in the
second line of \eqref{eq:M15}, when reduced modulo~$8$, except
for the term for $e=k+1$.

A similar partial cancellation, modulo~$8$, takes place between the
big terms in parentheses in the third and fourth lines of
\eqref{eq:M15}. 

For the first sum in the first line and the last sum in the last
line, we observe that
\begin{align*} %\label{}
4\sum_{i\ge k+1}&n_{2i+1}(n_{2i+2}+1)
+4\sum_{i\ge1}(n_{2i+1}+1)n_{2i+2}\\
&=
4\sum_{i\ge k+1}n_{2i+1}n_{2i+2}
+4\sum_{i\ge k+1}n_{2i+1}
+4\sum_{i\ge1}n_{2i+1}n_{2i+2}
+4\sum_{i\ge1}n_{2i+2}\\
&\equiv
4\sum_{i=1}^kn_{2i+1}n_{2i+2}
+4\sum_{i\ge 2k+2}n_{i}
+4\sum_{i=1}^kn_{2i+2}
\pmod8\\
&\equiv
4(k-2+n_{2k+1}n_{2k+2})
+4\sum_{i\ge 2k+2}n_{i}
+4(k-2+n_{2k+2})
\pmod8\\
&\equiv
4(n_{2k+1}+1)n_{2k+2}
+4\sum_{i\ge 2k+2}n_{i}
\pmod8.\\
\end{align*}
After further simplification, one obtains the claimed expression
$6 n_{K+1} + 4 n_{K+1} s_2(n) + 4$.

\medskip
Now let $K$ be odd, say $K=2k+1$. The relevant terms from which one
extracts the coefficient of $z^n$ are again the ones in
\eqref{eq:terms}. However, the relevant cases 
in Lemmas~\ref{lem:OmE} and \ref{lem:E} are not the same.
More precisely, here we obtain
\begin{multline} \label{eq:M15B}
4+4n_{2k+2}+4n_{2k+3}+4(n_{2k+2}+1)\sum_{i\ge 2k+2}n_i\\
+0
+ 6(2+2n_4)+2(n_4+n_5+3n_6+n_7+3n_8+n_9+\cdots)\\
+
4(1+2n_{2k+2}+n_{2k+3}+n_{2k+4}+3n_{2k+5}+n_{2k+6}+3n_{2k+7}+n_{2k+8}+\cdots)\\
+
6(2+n_{2k+2}+2n_{2k+3})+
4\sum_{e>k}(n_{2e+2}+n_{2e+3})(n_{2}+n_{3}+\dots+n_{2e+1})\\
+
4\sum_{i\ge1}(n_{2i+1}+1)n_{2i+2}+1.
\end{multline}
This expression can be simplified in much the same way as
\eqref{eq:M15}; with one notable difference though. The sums
in the next-to-last and the last lines have no counterparts
(as opposed to the situation in \eqref{eq:M15}). Consequently,
they need a different treatment. 

The sum in the next-to-last line of \eqref{eq:M15B} is first
transformed in the same way as before, see \eqref{eq:n'n}.
Then one observes that
\begin{align*} %\label{}
4\sum_{e>k}&(n_{2e+2}+n_{2e+3})(n_{2k+2}+n_{2k+3}+\dots+n_{2e+1})\\
&=4\sum_{i>j\ge 2k+2}n_in_j-4\sum_{i\ge k+1}n_{2i}n_{2i+1}\\
&=2\bigg(\sum_{i\ge 2k+2}n_i\bigg)^2-
2\sum_{i\ge 2k+2}n_i-4\sum_{i\ge k+1}n_{2i}n_{2i+1}\\
&=2\bigg(s_2(n)-\sum_{i=0}^{2k+1}n_i\bigg)^2-
2\sum_{i\ge 2k+2}n_i
-4\sum_{i\ge k+1}n_{2i}n_{2i+1}\\
&\equiv2s_2^2(n)-4n_{2k+1}s_2(n)+2n_{2k+1}-
2\sum_{i\ge 2k+2}n_i
-4\sum_{i\ge k+1}n_{2i}n_{2i+1}
\pmod8.
\end{align*}
The last sum combines with the sum $4\sum_{i\ge1}n_{2i+1}n_{2i+2}$
appearing in the last line of \eqref{eq:M15B} into $4e_2(n)$, up to some
error that can be computed explicitly modulo~$8$. 
We leave the remaining simplifications, leading to the claimed
expression 
$$ 2 n_{2k+2} + 4 n_{2k+2} s_2(n) + 2 s_2^2(n) + 4
 e_2(n) + 5,$$ 
to the reader.
\end{proof}

The following result of Eu, Liu and Yeh \cite{EuLYAB}, characterising
even congruence classes of~$M_n$ modulo~$8$, can now readily be
obtained through a straightforward case-by-case analysis using the
corresponding cases from Theorem~\ref{thm:M8}.

\begin{corollary} \label{thm:EuLYAB}
%The Motzkin number $M_n$ is divisible by $4$ if, and only if,
%$$
%n=(4i + 1)4^{j+1} - 1\quad \text{or}\quad  n = (4i + 3)4^{j+1} - 2,
%$$
%for some non-negative integers $i$ and $j$.
The Motzkin number $M_n$ is even if, and only if, 
$n = (4i + \ep)4^{j+1} -\de$ 
for non-negative integers $i, j$, $\ep = 1, 3$, and 
$\de = 1, 2$. Moreover, we have
$$
M_n\equiv_8\begin{cases} 
4,&\text{if\/ }(\ep,\de)=(1,1)\text{ or }(3,2),\\
4y+2,&\text{if\/ }(\ep,\de)=(1,2)\text{ or }(3,1),
\end{cases}
$$
where $y$ is the number of\/ $1$'s in the binary expansion of\/
$4i+\ep-1$.
\end{corollary} 

Moreover, Theorem~\ref{thm:M8} allows us to provide a characterisation
of all those~$n$ for which $M_n$ lies in {\it any} specified
congruence class modulo~$8$. As an example, we give the
characterisation for the congruence class~1.

\begin{corollary} \label{cor:Mn1}
The Motzkin number $M_n$ is congruent to~$1$ modulo~$8$ if, and only
if, $n$ satisfies one of the following conditions:

\begin{enumerate} 
\item $n\equiv0$~{\em(mod~$16$)} and
  $s_2(n)\equiv e_2(n)\equiv0$~{\em(mod~$2$);}
\item $n\equiv1$~{\em(mod~$16$),}
  $s_2(n)\equiv1$~{\em(mod~$2$),} and $e_2(n)\equiv0$~{\em(mod~$2$);}
\item $n\equiv4$~{\em(mod~$16$),}
  $s_2(n)\equiv1$~{\em(mod~$2$),} and $e_2(n)\equiv0$~{\em(mod~$2$);}
\item $n\equiv5$~{\em(mod~$16$),}
  $s_2(n)\equiv0$~{\em(mod~$2$),} and $e_2(n)\equiv1$~{\em(mod~$2$);}
\item $n\equiv6$~{\em(mod~$16$),} $n_4=0,$
  $s_2(n)\equiv e_2(n)\equiv1$~{\em(mod~$2$);}
\item $n\equiv6$~{\em(mod~$16$),} $n_4=1,$
  $s_2(n)\equiv e_2(n)\equiv0$~{\em(mod~$2$);}
\item $n\equiv7$~{\em(mod~$16$),} $n_4=0,$
  $s_2(n)\equiv0$~{\em(mod~$2$),} and $e_2(n)\equiv1$~{\em(mod~$2$);}
\item $n\equiv7$~{\em(mod~$16$),} $n_4=1,$
  $s_2(n)\equiv e_2(n)\equiv1$~{\em(mod~$2$);}
\item $n\equiv8$~{\em(mod~$16$),}
  $s_2(n)\equiv e_2(n)\equiv0$~{\em(mod~$2$);}
\item $n\equiv9$~{\em(mod~$16$),}
  $s_2(n)\equiv1$~{\em(mod~$2$),} and $e_2(n)\equiv0$~{\em(mod~$2$);}
\item $n\equiv12$~{\em(mod~$16$),}
  $s_2(n)\equiv1$~{\em(mod~$2$),} and $e_2(n)\equiv0$~{\em(mod~$2$);}
\item $n\equiv13$~{\em(mod~$16$),}
  $s_2(n)\equiv0$~{\em(mod~$2$),} and $e_2(n)\equiv1$~{\em(mod~$2$);}
\item $n\equiv14$~{\em(mod~$16$),} $K$ odd, $n_{K+1}=0,$
  $s_2(n)\equiv e_2(n)\equiv1$~{\em(mod~$2$);}
\item $n\equiv14$~{\em(mod~$16$),} $K$ odd, $n_{K+1}=1,$
  $s_2(n)\equiv e_2(n)\equiv0$~{\em(mod~$2$);}
\item $n\equiv15$~{\em(mod~$16$),} $K$ odd, $n_{K+1}=0,$
  $s_2(n)\equiv0$~{\em(mod~$2$),} and $e_2(n)\equiv1$~{\em(mod~$2$);}
\item $n\equiv15$~{\em(mod~$16$),} $K$ odd, $n_{K+1}=1,$
  $s_2(n)\equiv e_2(n)\equiv1$~{\em(mod~$2$);}
\end{enumerate}
where $s_2(n),$ $e_2(n),$ and $K$ are defined as in Theorem~{\em\ref{thm:M8}}.
\end{corollary}

It should be observed that it is straightforward to generate all
possible~$n$ in any of the 16~cases in the characterisation of the
previous corollary. 
%This seems less clear from the automaton
%in \cite[Fig.~4]{RoYaAA}.
Characterisations similar to the one in Corollary~\ref{cor:Mn1} 
are available for all other congruence classes modulo~$8$.

We leave it as an open problem whether a description for the odd
congruence classes of~$M_n$ modulo~$8$ exists which is
comparably compact as the one of Eu, Liu and Yeh for the even ones,
given here in Corollary~\ref{thm:EuLYAB}. We do not hide
that we are very sceptical about this.

\section{Further applications}
\label{Sec:FurthApp}

In this final section, we show that the same approach that we
applied in Sections~\ref{sec:Motzkin}--\ref{sec:Mn8} to Motzkin numbers
also works for the sequences of Motzkin prefix numbers, of Riordan
numbers, of hex tree numbers, and of central trinomial coefficients.
Since the arguments are completely analogous, we content ourselves
with brief sketches of the main points and subsequent statements
of the corresponding results for congruences modulo~$8$.

\subsection{Motzkin prefix numbers modulo $8$}

Let $MP_n$ be the $n$-th {\it  Motzkin prefix number}, that is, the number of
lattice paths from $(0,0)$ consisting of $n$ steps taken from
the set
$\{(1,0),(1,1),(1,-1)\}$ never running below the $x$-axis.
Gouyou-Beauchamps and Viennot \cite{GoViAA} have shown that
$MP_n$ also counts directed rooted animals with $n + 1$ vertices.
It is well-known (cf.\ e.g.\ \cite[Sec.~8]{KrMuAE})
that the generating function $MP(z)=\sum_{n\ge0}MP_n\,z^n$ satisfies the
functional equation
\begin{equation} \label{eq:MotzkinprefEF}
z(1-3z)MP^2(z)+(1-3z)MP(z)-1=0.
\end{equation}

When we apply our method from Section~\ref{sec:method},
the Ansatz for the base step is
\begin{equation} \label{eq:Ansatz-MP} 
MP_1(z)=\frac {1} {z}\Om^{2^\al}(z^4)+\sum_{k=0}^{\al+1}\frac {z^{2^k-1}}
{(1-z)^{2^k}}. 
\end{equation}
Subsequently, everything runs through in the same way as in
the proof of Theorem~\ref{thm:Motzkin}. Consequently, 
the generating function $MP(z)$ for Motzkin prefix numbers 
satisfies a completely analogous theorem.

If we now follow the line of argument in Section~\ref{sec:M8},
then we obtain that $MP(z)$ admits the following representation
modulo~$8$:
\begin{multline} \label{eq:MP}
MP(z)=\frac{4 E(z^4) \Om(z^4)}{z}+\left(\frac{1}{z}+\frac{4 z^7+2
   z^6+2 z^4+2 z^2+4
   z+2}{1-z^8}\right)
   \Om(z^4)\\
+\frac{4 E(z^8)}{z}+
   \left(\frac{6}{z}+\frac{4}{1-z^2}\right
   )E(z^4)\\
+\frac{\begin{matrix}6 z^{15}+z^{14}+z^{12}+4
   z^{11}+5 z^{10}+6 z^9+z^8\kern3cm\\\kern3cm +6 z^7+z^6+4
   z^5+z^4+4 z^3+5 z^2+2 z+1\end{matrix}}{1-z^{16}}
\quad \quad \text{modulo }8.
\end{multline}

Coefficient extraction using Lemmas~\ref{lem:E}--\ref{lem:OmE}
leads to the following theorem generalising \cite[Cor.~3.2]{DeSaAA}.

\begin{theorem} \label{thm:MP8}
Let $n$ be a positive integer with binary expansion
$$
n=n_0+n_1\cdot2+n_2\cdot2^2+\cdots.
$$
The Motzkin prefix numbers $MP_n$ satisfy the following congruences
modulo~$8${\em :}
$$MP_n\equiv_8 \begin{cases} 
2 s_2^2(n) + 4 e_2(n) + 1,&\text{if\/ }n\equiv0~\text{\em(mod $16$),}\\
4 s_2(n) + 6,
&\text{if\/ }n\equiv1~\text{\em(mod $16$),}\\
6 s_2^2(n) + 4 e_2(n) + 7,&\text{if\/ }n\equiv2~\text{\em(mod $16$),}\\
2 s_2^2(n) + 4 e_2(n) + 1,&\text{if\/ }n\equiv3~\text{\em(mod $16$),}\\
2 s_2^2(n) + 4 e_2(n) + 1,&\text{if\/ }n\equiv4~\text{\em(mod $16$),}\\
0,&\text{if\/ }n\equiv5~\text{\em(mod $16$),}\\
6 s_2^2(n) + 4 e_2(n) + 7,
&\text{if\/ }n\equiv6~\text{\em(mod $16$),}\\
4 s_2(n) + 2 n_4 + 4 n_4 s_2(n) + 2,
&\text{if\/ }n\equiv7~\text{\em(mod $16$),}\\
2 s_2^2(n) + 4 e_2(n) + 1,&\text{if\/ }n\equiv8~\text{\em(mod $16$),}\\
4 s_2(n) + 6,
&\text{if\/ }n\equiv9~\text{\em(mod $16$),}\\
6 s_2^2(n) + 4 e_2(n) + 7,&\text{if\/ }n\equiv10~\text{\em(mod $16$),}\\
6 s_2^2(n) + 4 e_2(n) + 7,&\text{if\/ }n\equiv11~\text{\em(mod $16$),}\\
2 s_2^2(n) + 4 e_2(n) + 1,
&\text{if\/ }n\equiv12~\text{\em(mod $16$),}\\
0,&\text{if\/ }n\equiv13~\text{\em(mod $16$),}\\
6 s_2^2(n) + 4 e_2(n) + 7, 
&\text{if\/ }n\equiv14~\text{\em(mod $16$),}\\
2 s_2^2(n) + 4 e_2(n) + 6 n_{K+1} + 4 n_{K+1} s_2(n) + 1,\hskip-2cm\\
&\text{if\/ }n\equiv15~\text{\em(mod $16$) and $K$ is even,}\\ 
4 s_2(n) + 2 n_{K+1} + 4 n_{K+1} s_2(n) + 2,\hskip-1cm\\
&\text{if\/ }n\equiv15~\text{\em(mod $16$) and $K$ is odd},
\end{cases}
$$
where $s_2(n),$ $e_2(n),$ and $K$ are defined as in Theorem~{\em\ref{thm:M8}}.
\end{theorem}

\subsection{Riordan numbers modulo $8$}

Let $R_n$ be the $n$-th {\it Riordan number}, that is, the number of
lattice paths from $(0,0)$ to $(n,0)$ consisting of steps taken from
the set
$\{(1,0),(1,1),(1,-1)\}$ never running below the $x$-axis, and where
steps $(1,0)$ are not allowed on the $x$-axis.
It is well-known (cf.\ e.g.\ \cite[Sec.~9]{KrMuAE})
that the generating function $R(z)=\sum_{n\ge0}R_n\,z^n$ satisfies the
functional equation
\begin{equation} \label{eq:RiordanEF}
z(1+z)R^2(z)-(z+1)R(z)+1=0.
\end{equation}

When we apply our method from Section~\ref{sec:method},
the Ansatz for the base step is the same as the one in the case of
Motzkin prefix numbers, that is, 
the right-hand side of \eqref{eq:Ansatz-MP}.
Subsequently, everything runs through in the same way as in
the proof of Theorem~\ref{thm:Motzkin}. Consequently, 
the generating function $R(z)$ for Riordan numbers 
satisfies a completely analogous theorem.

If we now follow the line of argument in Section~\ref{sec:M8},
then we obtain that $R(z)$ admits the following representation
modulo~$8$:
\begin{multline} \label{eq:R}
R(z)=
\frac{4 E(z^4)\Om(z^4)}{z} +\left(\frac{7}{z}+\frac{2 z^6+4
   z^5+2 z^4+4 z^3+2
   z^2+2}{1-z^8}\right)\Om(z^4) \\
+\frac{4 E(z^8)
   }{z}+
   \left(\frac{2}{z}+\frac{4}{1-z^2}\right)E(z^4)\\
   +\frac{\begin{matrix}5 z^{14}+2 z^{13}+z^{12}+2
   z^{11}+z^{10}+4 z^9+z^8\kern3cm\\\kern5cm +5 z^6+6
   z^5+z^4+2 z^3+z^2+1\end{matrix}}{1-z^{16}}
\quad \quad \text{modulo }8.
\end{multline}

Coefficient extraction using Lemmas~\ref{lem:E}--\ref{lem:OmE}
leads to the following theorem generalising \cite[Cor.~3.3]{DeSaAA}.

\begin{theorem} \label{thm:R8}
Let $n$ be a positive integer with binary expansion
$$
n=n_0+n_1\cdot2+n_2\cdot2^2+\cdots.
$$
The Riordan numbers $R_n$ satisfy the following congruences
modulo~$8${\em :}
$$R_n\equiv_8 \begin{cases} 
2 s_2^2(n) + 4 e_2(n) + 1,&\text{if\/ }n\equiv0~\text{\em(mod $16$),}\\
0,
&\text{if\/ }n\equiv1~\text{\em(mod $16$),}\\
6 s_2^2(n) + 4 e_2(n) + 3,&\text{if\/ }n\equiv2~\text{\em(mod $16$),}\\
2 s_2^2(n) + 4 e_2(n) + 5,&\text{if\/ }n\equiv3~\text{\em(mod $16$),}\\
2 s_2^2(n) + 4 e_2(n) + 1,&\text{if\/ }n\equiv4~\text{\em(mod $16$),}\\
4 s_2(n) + 6,&\text{if\/ }n\equiv5~\text{\em(mod $16$),}\\
6 s_2^2(n) + 4 e_2(n) + 3,
&\text{if\/ }n\equiv6~\text{\em(mod $16$),}\\
6 n_4 + 4 n_4 s_2(n) + 4,
&\text{if\/ }n\equiv7~\text{\em(mod $16$),}\\
2 s_2^2(n) + 4 e_2(n) + 1,&\text{if\/ }n\equiv8~\text{\em(mod $16$),}\\
0,
&\text{if\/ }n\equiv9~\text{\em(mod $16$),}\\
6 s_2^2(n) + 4 e_2(n) + 3,&\text{if\/ }n\equiv10~\text{\em(mod $16$),}\\
6 s_2^2(n) + 4 e_2(n) + 7,&\text{if\/ }n\equiv11~\text{\em(mod $16$),}\\
2 s_2^2(n) + 4 e_2(n) + 1,
&\text{if\/ }n\equiv12~\text{\em(mod $16$),}\\
4 s_2(n) + 6,&\text{if\/ }n\equiv13~\text{\em(mod $16$),}\\
6 s_2^2(n) + 4 e_2(n) + 3, 
&\text{if\/ }n\equiv14~\text{\em(mod $16$),}\\
  2 s_2^2(n) + 4 e_2(n) + 2 n_{K+1} + 4 n_{K+1} s_2(n) + 5,\hskip-2cm\\
&\text{if\/ }n\equiv15~\text{\em(mod $16$) and $K$ is even,}\\ 
6 n_{K+1} + 4 n_{K+1} s_2(n) + 4,
&\text{if\/ }n\equiv15~\text{\em(mod $16$) and $K$ is odd},
\end{cases}
$$
where $s_2(n),$ $e_2(n),$ and $K$ are defined as in Theorem~{\em\ref{thm:M8}}.
\end{theorem}

\subsection{Hex tree numbers modulo $8$}

Let $H_n$ be the $n$-th {\it hex tree number}, that is, 
the number of planar rooted trees where each vertex may have a left, a
middle, or a right descendant, but never a left {\it and\/} middle
descendant, and never a middle {\it and\/} right descendant.
It is well-known (and easy to see)
that the generating function $H(z)=\sum_{n\ge0}H_n\,z^n$
satisfies the functional equation
\begin{equation} \label{eq:hex treeEF}
z^2H^2(z)+(3z-1)H(z)+1=0.
\end{equation}

When we apply our method from Section~\ref{sec:method},
the Ansatz for the base step is
\begin{equation*} %\label{eq:Ansatz-H} 
H_1(z)=\frac {1-z} {z^2}\Om^{2^\al}(z^4)+\sum_{k=1}^{\al+1}\frac {z^{2^k-2}}
{(1-z)^{2^k-1}}. 
\end{equation*}
Subsequently, everything runs through in the same way as in
the proof of Theorem~\ref{thm:Motzkin}. Consequently, 
the generating function $H(z)$ for hex tree numbers 
satisfies a completely analogous theorem.

If we now follow the line of argument in Section~\ref{sec:M8},
then we obtain that $H(z)$ admits the following representation
modulo~$8$:
\begin{multline} \label{eq:H}
H(z)=
\frac{4 E(z^4)\Om(z^4)
   (z+1)}{z^2}\\
+ \left(\frac{7}{z^2}+\frac{3}{z}+\frac{6 z^7+2
   z^6+2 z^5+6 z^4+2 z^3+6 z^2+6
   z+2}{1-z^8}\right)\Om(z^4)\\
+\frac{4
   E(z^8) (z+1)}{z^2}
+
   \left(\frac{2}{z^2}+\frac{2}{z}+\frac{4}
   {1-z}\right)E(z^4)\\
+\frac{\begin{matrix}3 z^{15}+5 z^{14}+5
   z^{13}+3 z^{12}+z^{11}+3 z^{10}+7 z^9+5
   z^8\kern1cm\\\kern2.5cm +3 z^7+5 z^6+z^5+7 z^4+z^3+3 z^2+3
   z+1\end{matrix}}{1-z^{16}}
\quad \quad \text{modulo }8.
\end{multline}

Coefficient extraction using Lemmas~\ref{lem:E}--\ref{lem:OmE}
leads to the following theorem generalising \cite[Cor.~3.4]{DeSaAA}.

\begin{theorem} \label{thm:H8}
Let $n$ be a positive integer with binary expansion
$$
n=n_0+n_1\cdot2+n_2\cdot2^2+\cdots.
$$
The hex tree numbers $H_n$ satisfy the following congruences
modulo~$8${\em :}
$$H_n\equiv_8 \begin{cases} 
2 s_2^2(n) + 4 e_2(n) + 1,&\text{if\/ }n\equiv0~\text{\em(mod $16$),}\\
2 s_2^2(n) + 4 e_2(n) + 1,
&\text{if\/ }n\equiv1~\text{\em(mod $16$),}\\
4 s_2(n) + 6,&\text{if\/ }n\equiv2~\text{\em(mod $16$),}\\
4,&\text{if\/ }n\equiv3~\text{\em(mod $16$),}\\
6 s_2^2(n) + 4 e_2(n) + 3,&\text{if\/ }n\equiv4~\text{\em(mod $16$),}\\
6 s_2^2(n) + 4 e_2(n) + 7,&\text{if\/ }n\equiv5~\text{\em(mod $16$),}\\
6 s_2^2(n) + 4 e_2(n) + 2 n_4 + 4 n_4 s_2(n) + 7,
&\text{if\/ }n\equiv6~\text{\em(mod $16$),}\\
6 s_2^2(n) + 4 e_2(n) + 6 n_4 + 4 n_4\, s_2(n) + 7,
&\text{if\/ }n\equiv7~\text{\em(mod $16$),}\\
2 s_2^2(n) + 4 e_2(n) + 1,&\text{if\/ }n\equiv8~\text{\em(mod $16$),}\\
2 s_2^2(n) + 4 e_2(n) + 1,
&\text{if\/ }n\equiv9~\text{\em(mod $16$),}\\
4,&\text{if\/ }n\equiv10~\text{\em(mod $16$),}\\
4 s_2(n) + 6,&\text{if\/ }n\equiv11~\text{\em(mod $16$),}\\
6 s_2^2(n) + 4 e_2(n) + 3,
&\text{if\/ }n\equiv12~\text{\em(mod $16$),}\\
6 s_2^2(n) + 4 e_2(n) + 7,&\text{if\/ }n\equiv13~\text{\em(mod $16$),}\\
   4 s_2(n) + 6 n_{K+1} + 4 n_{K+1} s_2(n) + 6, 
&\text{if\/ }n\equiv14~\text{\em(mod $16$) and $K$ is even,}\\
6 s_2^2(n) + 4 e_2(n) + 2 n_{K+1} + 4 n_{K+1} s_2(n) + 7,\hskip-3cm\\
&\text{if\/ }n\equiv14~\text{\em(mod $16$) and $K$ is odd,}\\
2 n_{K+1} + 4 n_{K+1} s_2(n) + 4,
&\text{if\/ }n\equiv15~\text{\em(mod $16$) and $K$ is even,}\\ 
6 s_2^2(n) + 4 e_2(n) + 6 n_{K+1} + 4 n_{K+1} s_2(n) + 7,\hskip-1cm\\
&\text{if\/ }n\equiv15~\text{\em(mod $16$) and $K$ is odd},
\end{cases}
$$
where $s_2(n),$ $e_2(n),$ and $K$ are defined as in Theorem~{\em\ref{thm:M8}}.
\end{theorem}

Inspection of the above sixteen cases reveals the interesting
fact that hex tree numbers are never divisible by~$8$
(a fact that is also true for Motzkin numbers).

\subsection{Central trinomial coefficients modulo $8$}

Let $T_n$ be the $n$-th {\it central trinomial coefficient}, that is, the
coefficient of $t^n$ in $(1+t+t^2)^n$.
Already Euler knew 
that the generating function $T(z)=\sum_{n\ge0}T_n\,z^n$
equals 
\begin{equation} \label{eq:zentrin} 
T(z)=(1-2z-3z^2)^{-1/2}
\end{equation}
 (cf.\ \cite[solution to 
Exercise~6.42]{StanBI}). 

Here we proceed differently in order to express $T(z)$ in terms of our
basic series $\Om(z^4)$ (and the error series $E(z^4)$).
We are forced to do so since our method from
Section~\ref{sec:method} fails.\footnote{The base step solution
would be the one with $a_0(z)=1/(1-z)$ and $a_{2^\al}(z)=0$.
Subsequently, the iteration step modulo~$4$ would not succeed,
however. We believe that the reason for this phenomenon might be that
the minimal polynomial for $\Om$ (in the sense of
\cite{KaKMAA,KrMuAE,KrMuAL}) is of degree less than~$4$.}
Instead, we compare \eqref{eq:Motzkin} and
\eqref{eq:zentrin} to see that\footnote{A similar approach would also
have been possible for the Motzkin 
prefix and the Riordan numbers since their generating functions
satisfy relations with the generating function $M(z)$ for Motzkin
numbers analogous to \eqref{eq:TM}. However, no such relation 
exists between the generating function for hex tree numbers
and $M(z)$.}
\begin{equation} \label{eq:TM} 
T(z)=\frac {1-z-2z^2M(z)} {1-2z-3z^3}.
\end{equation}
Relation \eqref{eq:TM} would allow us to establish a complete analogue
of Theorem~\ref{thm:Motzkin} for the generating
function $T(z)$ for central trinomial coefficients.
In order to determine these numbers modulo~$8$,
we substitute \eqref{eq:MM} in the above relation. Then
some simplification eventually leads to
\begin{multline} \label{eq:T}
T(z)=
\frac{6 z^3+6 z^2+2 z+2
   }{1-z^4}\Om(z^4)+\frac{4 E(z^4)}{1-z}\\
+\frac{3
   z^7+7 z^6+z^5+z^4+7 z^3+3
   z^2+z+1}{1-z^8}
\quad \quad \text{modulo }8.
\end{multline}

Coefficient extraction using Lemmas~\ref{lem:E1-z} and \ref{lem:Omz^2}
leads to the following theorem for central trinomial coefficients modulo~$8$.

\begin{theorem} \label{thm:T8}
Let $n$ be a positive integer with binary expansion
$$
n=n_0+n_1\cdot2+n_2\cdot2^2+\cdots.
$$
Modulo~$8,$ the central trinomial coefficients $T_n$ 
satisfy the following congruences:
$$T_n\equiv_8 \begin{cases} 
2 s_2^2(n) + 4 e_2(n) + 1,
&\text{if\/ }n\equiv0~\text{\em(mod $2$),}\\
6 s_2^2(n) + 4 e_2(n) + 3,
&\text{if\/ }n\equiv1~\text{\em(mod $2$),}\\
\end{cases}
$$
where $s_2(n)$ and $e_2(n)$ are defined as in Theorem~{\em\ref{thm:M8}}.
\end{theorem}

\section*{Note}
The result in Theorem~\ref{thm:M8} has in the meantime also been obtained
--- in slightly different form and using a different approach --- by Wang and Xin
in ``A Classification of Motzkin Numbers Modulo~$8$" [{Electron\@. J. Combin\@.}
{\bf 25}(1) (2018), Article~P1.54, 15~pp.].

\section*{Acknowledgement}
The authors are indebted to an anonymous referee for clarifying
the scope of the automaton method of Rowland and Yassawi in \cite{RoYaAA}.

\end{document}